\documentclass[english]{amsart}

\usepackage{esint}
\usepackage[svgnames]{xcolor} 
\usepackage[colorlinks,citecolor=red,pagebackref,hypertexnames=false,breaklinks]{hyperref}
\usepackage{pgf,tikz}
\usepackage{pdfsync}

\usepackage{dsfont}
\usepackage{url}
\usepackage[utf8]{inputenc}
\usepackage[T1]{fontenc}
\usepackage{lmodern}
\usepackage{babel}
\usepackage{mathtools}  
\usepackage{amssymb}
\usepackage{lipsum}
\usepackage{mathrsfs}
\usepackage{color}
\usepackage[marginparwidth=2cm]{geometry}

\newtheorem{theorem}{Theorem}[section]
\newtheorem{proposition}{Proposition}[section]
\newtheorem{lemma}{Lemma}[section]

\newtheorem{corollary}{Corollary}[section]

\numberwithin{equation}{section}

\newcommand{\be}{\begin{equation}}
\newcommand{\ee}{\end{equation}}
\newcommand{\ba}{\begin{array}}
\newcommand{\ea}{\end{array}}
\newcommand{\bea}{\begin{eqnarray*}}
\newcommand{\eea}{\end{eqnarray*}}
\newcommand{\bean}{\begin{eqnarray}}
\newcommand{\eean}{\end{eqnarray}}

\title[qualitative photoacoustic tomography]{Stability for quantitative photoacoustic 
tomography revisited}

\thanks{The authors were supported by the grant ANR-17-CE40-0029 of the French National Research Agency ANR (project MultiOnde).}
 
\author[Eric Bonnetier]{Eric Bonnetier}
\address{Fourier Institute,  Universit\'e  
Grenoble-Alpes, 700 Avenue Centrale,
38401 Saint-Martin-d'H\`eres, France}
\email{eric.bonnetier@univ-grenoble-alpes.fr}

\author[Mourad Choulli]{Mourad Choulli}
\address{Universit\'e de Lorraine}
\email{mourad.choulli@univ-lorraine.fr}

\author[Faouzi Triki]{Faouzi Triki}
\address{Laboratoire Jean Kuntzmann,  UMR CNRS 5224, 
Universit\'e  Grenoble-Alpes, 700 Avenue Centrale,
38401 Saint-Martin-d'H\`eres, France}
\email{faouzi.triki@univ-grenoble-alpes.fr}

\date{}

\begin{document}

\begin{abstract}

 This paper deals with the issue of stability  in determining
the absorption  and the diffusion coefficients in  quantitative photoacoustic imaging. 
We  establish a global conditionnal H\"older  stability inequality  from the knowledge of two internal data obtained from  optical waves, generated by two point sources  in a region where the optical coefficients are known.


\end{abstract}

\maketitle

\tableofcontents

\section{Introduction}\label{section1}

Throughout this text $n\ge 3$ is a fixed integer. If $0<\beta \le 1$ we denote by $C^{0,\beta} (\mathbb{R}^n)$ the vector space of bounded continuous functions $f$ on $\mathbb{R}^n$ satisfying
\[
[f]_\beta=\sup\left\{ \frac{|f(x)-f(y)|}{|x-y|^\beta} ;\; x,y\in \mathbb{R}^n,\; x\ne y\right\}<\infty .
\]
$C^{0,\beta} (\mathbb{R}^n)$ is then a Banach space when it is endowed with its natural norm
\[
\|f\|_{C^{0,\beta} (\mathbb{R}^n)}=\|f\|_{L^\infty(\mathbb{R}^n)}+[f]_\beta .
\]
Define $C^{1,\beta} (\mathbb{R}^n)$ as the vector space of functions $f$ from $C^{0,\beta} (\mathbb{R}^n)$ so that $\partial_jf\in  C^{0,\beta} (\mathbb{R}^n)$, $1\le j\le n$. The vector space $C^{1,\beta} (\mathbb{R}^n)$ equipped with the norm
\[
\|f\|_{C^{1,\beta} (\mathbb{R}^n)}=\|f\|_{C^{0,\beta} (\mathbb{R}^n)}+\sum_{j=1}^n\|\partial_jf\|_{C^{0,\beta} (\mathbb{R}^n)}
\]
is a Banach space.

The data in this paper consists in $\xi_1,\xi_2\in \mathbb{R}^n$, $\Omega \Subset \mathbb{R}^n\setminus\{\xi_1,\xi_2\}$ a $C^{1,1}$ bounded domain with boundary $\Gamma$, $0<\alpha <1$, $0<\theta <\alpha$, $\lambda > 1$ and  $\kappa > 1$. For notational convenience, the set of data will denoted by $\mathfrak{D}$. That is
\[
\mathfrak{D}=(n,\xi_1,\xi_2,\Omega ,\alpha ,\theta , \lambda ,\kappa) .
\]
Denote by $\mathcal{D}(\lambda,\kappa)$ the set of couples $(a,b)\in C^{1,1} (\mathbb{R}^n)\times C^{0,1} (\mathbb{R}^n)$ satisfying 
\bean
\lambda ^{-1}\le a\quad \mbox{and}\quad \|a\|_{C^{1,1} (\mathbb{R}^n)}\le \lambda,\\
\kappa ^{-1}\le b\quad \mbox{and} \quad \|b\|_{C^{0,1} (\mathbb{R}^n)}\le \kappa,
\eean
Define further the elliptic operator ${L}_{a,b}$ acting as follows
\begin{equation}\label{i1}
{L}_{a,b}u(x)=-\mathrm{div}(a(x)\nabla u(x))+b(x)u(x).
\end{equation}
We show in Section \ref{section2} that if $(a,b)\in \mathcal{D}(\lambda, \kappa)$ then the operator $L_{a,b}$ admits a unique fundamental solution $G_{a,b}$ satisfying, where $\xi \in \mathbb{R}^n$,
\[
G_{a,b}(\cdot ,\xi )\in C^{2,\alpha}_{\mathrm loc}(\mathbb{R}^n\setminus\{\xi\}), \quad L_{a,b} G_{a,b}(\cdot ,\xi )=0\;  \mbox{in}\; \mathbb{R}^n\setminus\{\xi\},
\]
and,  for any $f\in C_0^\infty (\mathbb{R}^n ),$ 
\[
u=\int_{\mathbb{R}^n}G_{a,b}(\cdot ,\xi)f(\xi )d\xi
\]
belongs to $H^2(\mathbb{R}^n)$ and it is the unique solution of $L_{a,b} u=f$. 
 
We deal in the present work with the problem of reconstructing $(a,b)\in \mathcal{D}(\lambda, \kappa)$ from energies generated by two point sources located at $\xi_1$ and $\xi_2$. Precisely, if $u_j(a,b)=G_{a,b}(\cdot ,\xi_j)$, $j=1,2$, we want to determine $(a,b)$ from the internal measurements 
\[ 
v_j(a,b)=bu_j(a,b)\quad  \textrm{in}\; \Omega,\quad j=1, 2.
\]

This inverse problem is related to photoacoustic tomography (PAT) where  optical energy absorption causes thermoelastic expansion of the tissue, which  in turn generates a pressure wave \cite{Wang-Book09}. This  acoustic signal is measured by transducers distributed on the boundary of the sample and  it is used for imaging optical properties of the sample.   The internal data $v_1(a,b)$ and $v_2(a,b)$ are obtained by performing  a first step consisting in a linear initial to boundary inverse problem for the acoustic wave equation.  Therefore the inverse problem that arises from this first inversion is to determine the diffusion coefficient $a$ and the absorption coefficient $b$ from the internal data $v_1(a,b)$ and $v_2(a,b)$ that are proportional to the local  absorbed optical energy inside the sample. This inverse problem is known in the literature as quantitative photoacoustic tomography \cite{ADFV, AKK, AmBoJuKa-SIAM10, AGKNS, BaUh-IP10, BR1, NaSc-SIAM14, Ch2}. 

Photoacoustic imaging provides in theory images of optical contrasts and ultrasound resolution \cite{Wang-Book09}. Indeed,  the resolution is mainly due to the small wavelength  of acoustic waves, while the contrast is  somehow related to  the sensitivity of optical waves to absorption and scattering properties of the medium in the diffusive regime.  However, in practice, it has been observed in various experiments that the imaging depth, i.e. the maximal depth of the medium at which structures can be resolved at expected resolution, of (PAT) is still fairly limited, usually on the order of millimeters. This is mainly due to the  fact that  optical waves are  significantly attenuated by absorption and scattering. In fact the generated 
 optical signal decays very fast in the depth direction. This is indeed a well known faced  issue in optical tomography \cite{Wa}.  In most physicists works dealing with quantitative (PAT), the absorption  $b>0$ is assumed to be constant and  the optical wave  is simplified to  $Ce^{-b z}$,  as a function of  the depth $z$,  which is known as  the Beer-Lambert-Bouguer law \cite{Cox-LA}. Recently in \cite{RT}, assuming that medium is layered, the authors  derived a stability estimate that shows that  the reconstruction of the optical coefficients  is stable in the region close to the optical illumination source and deteriorate exponentially far away.

Stability inequalities for this inverse problem  were  first obtained in \cite{BR1,BaUh-IP10} under a  strong non-degeneracy assumption.   Later in  \cite{ADFV}, the authors improved these results by removing the non-degeneracy assumption for well-chosen boundary conditions (Definition 2.3). 

Assuming that the optical waves  are generated  by two point sources $\delta_{\xi_i}, i=1,2$,  we aim to derive a stability estimate for  the recovery of the optical coefficients from internal data. We point out that taking the optical wave  generated by a point source outside the sample seems to be more realistic  than assuming a boundary condition.

In the statement of Theorem \ref{theorem1} below $C=C(\mathfrak{D})>0$ and $0<\gamma=\gamma (\mathfrak{D}) <1$ are constants.

\begin{theorem}\label{theorem1} 
For any $(a,b), (\tilde{a},\tilde{b})\in \mathcal{D}(\lambda, \kappa)$ satisfying $(a,b)= (\tilde{a},\tilde{b})$ on $\Gamma$, we have
\[
\|a -\tilde{a}\|_{C^{1,\alpha}(\overline{\Omega})}+ \|b -\tilde{b}\|_{C^{0,\alpha}(\overline{\Omega})} \le C\left(\|v_1 -\tilde{v}_1\|_{C(\overline{\Omega}) }+\|v_2 -\tilde{v}_2\|_{C(\overline{\Omega})}\right)^\gamma .
\]
\end{theorem}

The rest of this text is organized as follows. In section 2 we construct a fundamental solution and give its regularity induced by that of the coefficients of the operator under consideration. We derive pointwise lower and upper bounds
for the fundamental solution that are of interest themselves. These bounds show how the optical signal
decays fast  in the depth direction.   We also establish in this section a lower bound of the local $L^2$-norm of the gradient of the quotient of two fundamental solutions near one of the point sources. This is the key point for establishing our stability inequality. This last result is then used in Section 3 to obtain a uniform polynomial lower bound of the local $L^2$-norm of the gradient in a given region. This polynomial lower bound is obtained in two steps. In the first step we derive, via a three-ball inequality for the gradient, a uniform lower bound of negative exponential type. We use then in the second  step an argument based on  the so-called frequency function in order to improve this lower bound. In the last section we prove our main theorem following the known method consisting in reducing the original problem to the stability of an inverse conductivity problem.

\section{Fundamental solutions}\label{section2}

\subsection{Constructing fundamental solutions}\label{sectionE}

In this subsection we construct a fundamental solution of divergence form elliptic operators. Since our construction relies on heat kernel estimates, we first recall some known results.

Consider the parabolic operator $P_{a,b}$ acting as follows
\[
P_{a,b}u(x,t) = -L_{a, b}u(x,t) - \partial_t u(x,t)
\]
and set

\[
Q=\{ (x,t,\xi ,\tau )\in \mathbb{R}^n\times \mathbb{R}\times \mathbb{R}^n\times \mathbb{R};\; \tau <t\}.
\]
Recall that a fundamental solution of the operator  $P_{a, b}$  is a function $E_{a,b}\in C^{2,1}(Q)$ verifying $P_{a,b}E=0$ in $Q$ and, for every $f\in C_0^\infty (\mathbb{R}^n)$,
\[
\lim_{t\downarrow \tau} \int_{\mathbb{R}^n}E_{a,b}(x,t,\xi ,\tau )f(\xi)d\xi =f(x),\quad x\in \mathbb{R}^n.
\]
The classical results in the monographs by A. Friedman \cite{Fr}, O. A. Ladyzenskaja, V. A. Solonnikov  and N.N Ural'ceva \cite{LSU} show that $P_{a,b}$ admits a non negative fundamental solution when $(a,b)\in \mathcal{D}(\lambda ,\kappa )$.

It is worth mentioning that if $a=c$, for some constant $c>0,$ and $b=0$ then the fundamental solution $E_{c,0}$ is explicitly given by
\[
 E_{c,0}(x,t,\xi , \tau )=\frac{1}{[4\pi c(t-\tau)]^{n/2}}e^{-\frac{|x-\xi|^2}{4c(t-\tau)}},\quad (x,t,\xi,\tau)\in Q.
\]
Examining carefully the proof of the two-sided Gaussian bounds in \cite{FS}, we see that these bounds remain valid whenever  $a\in C^{1,1}(\mathbb{R}^n)$  satisfies 
\begin{equation}\label{de}
\lambda ^{-1}\le a\quad \mbox{and}\quad \|a\|_{C^{1,1} (\mathbb{R}^n)}\le \lambda .
\end{equation}
More precisely we have the following theorem in which
\[
\mathcal{E}_c(x,t)=\frac{c}{t^{n/2}}e^{-\frac{|x|^2}{ct}},\quad x\in \mathbb{R}^n,\; t>0,\; c>0.
\]

\begin{theorem}\label{theorem-FS}
 There exists a constant $c=c(n,\lambda )>1$ so that, for any $a\in C^{1,1}(\mathbb{R}^n)$ satisfying \eqref{de}, we have
\begin{equation}\label{gb2}
\mathcal{E}_{c^{-1}}(x-\xi ,t-\tau) \le E_{a,0}(x,t;\xi,\tau) \le \mathcal{E}_c(x-\xi ,t-\tau),
\end{equation}
for all $(x,t,\xi ,\tau )\in Q$.
\end{theorem}

The relationship between $\mathcal{E}_c$ and $E_{c,0}$ is given by the formula
\begin{equation}\label{ee}
\mathcal{E}_c(x-\xi ,t-\tau)=\frac{(\pi c)^{n/2+1}}{\pi}E_{c/4,0}(x,t,\xi , \tau ),\quad (x,t,\xi ,\tau )\in Q.
\end{equation}
The following comparison principle will be useful in the sequel.

\begin{lemma}\label{comparaison-lemma}
Let $(a,b_1),(a,b_2)\in \mathcal{D}(\lambda ,\kappa)$ so that $b_1\le b_2$. Then $E_{a,b_2}\le E_{a,b_1}$.
\end{lemma}

\begin{proof}
Pick $0\le f\in C_0^\infty (\mathbb{R}^n)$. Let $u$ be the solution of the initial value problem
\[
P_{a,b_1}u(x,t)=0\; \in \mathbb{R}^n\times \{t>\tau\}, \quad u(x,\tau )=f.
\]
We have
\begin{equation}\label{comparaison1}
u(x,t)=\int_{\mathbb{R}^n}E_{a,b_1} (x,t;\xi, \tau )f(\xi) d\xi \ge 0.
\end{equation}
On the other hand, as $P_{a,b_1}u(x,t)=0$ can be rewritten as 
\[
P_{a,b_2}u(x,t)=[b_1(x)-b_2(x)]u(x,t),
\]
we obtain
\begin{align}
u(x,t)= \int_{\mathbb{R}^n}&E_{a,b_2}(x,t;\xi, \tau )f(\xi) d\xi \label{comparaison2}
\\
&-\int_\tau ^t \int_{\mathbb{R}^n} E_{a,b_2}(x,t;\xi,s )[b_1(\xi )-b_2(\xi)]u(\xi ,s)d\xi ds.\nonumber
\end{align}
Combining \eqref{comparaison1} and \eqref{comparaison2}, we get
\[
\int_{\mathbb{R}^n}E_{a,b_2} (x,t;\xi, \tau )f(\xi) d\xi \le \int_{\mathbb{R}^n}E_{a,b_1} (x,t;\xi, \tau )f(\xi) d\xi ,
\]
which yields in a straightforward manner the expected inequality.
\end{proof}

Consider, for $(a,b)\in \mathcal{D}(\lambda ,\kappa )$, the unbounded operator $A_{a,b}:L^2(\mathbb{R}^n)\rightarrow L^2(\mathbb{R}^n)$ defined
\[
A_{a,b}=-L_{a,b},\quad D(A_{a,b})=H^2(\mathbb{R}^n).
\]
It is well known that $A_{a,b}$ generates an analytic semigroup $e^{tA_{a,b}}$. Therefore in light of \cite[Theorem 4 on page 30, Theorem 18 on page 44 and the proof in the beginning of Section 1.4.2 on page 35]{AT} $k_{a,b}(t,x;\xi )$, the Schwarz kernel of $e^{tA_{a,b}}$, is H\"older continuous with respect to $x$ and $\xi$, satisfies 
\begin{equation}\label{ugb}
|k_{a,b}(t,x,\xi )|\le e^{-\delta t}\mathcal{E}_c(x-\xi ,t)
\end{equation}
and, for $|h|\le \sqrt{t}+|x-\xi|$,
\begin{align}
&|k_{a,b}(t,x+h,\xi )-k_{a,b}(t,x,\xi )|\le e^{-\delta t}\left( \frac{|h|}{\sqrt{t}+|x-\xi|}\right)^\eta \mathcal{E}_c(x-\xi ,t),\label{holder1}
\\
&|k_{a,b}(t,x,\xi +h)-k_{a,b}(t,x,\xi )|\le e^{-\delta t} \left( \frac{|h|}{\sqrt{t}+|x-\xi|}\right)^\eta \mathcal{E}_c(x-\xi ,t),\label{holder2}
\end{align}
where $c=c(n,\lambda,\kappa)>0$ and $\delta=\delta(n,\lambda,\kappa) >0$  and $\eta>0$ are constants.

From the uniqueness of solutions of the Cauchy problem
\begin{equation}\label{CP}
u'(t)=A_{a,b}u(t),\; t>0,\quad u(0)=f\in C_0^\infty (\mathbb{R}^n),
\end{equation}
we deduce in a straightforward manner that $k_{a,b}(t,x;\xi )=E_{a,b}(x,t;\xi ,0)$.

Prior to giving the construction of the fundamental solution for the variable coefficients operators, we state a result for operators with constant coefficients. This result is proved in Appendix \ref{appendixA}.

\begin{lemma}\label{lemma-tse}
Let $\mu >0$ and $\nu >0$ be two constants. Then the fundamental solution for the operator $-\mu \Delta +\nu$ is given by $G_{\mu ,\nu}(x,\xi )=\mathcal{G}_{\mu ,\nu}(x-\xi)$, $x,\xi \in \mathbb{R}^n$, with
\[
\mathcal{G}_{\mu ,\nu}(x)= (2\pi \mu)^{-n/2}(\sqrt{\nu \mu}/|x|)^{n/2-1}K_{n/2-1}(\sqrt{\nu}|x|/\sqrt{\mu}),\quad x\in \mathbb{R}^n.
\]
Here $K_{n/2-1}$ is the usual modified Bessel function of second kind. Moreover the following two-sided inequality holds
\begin{equation}\label{tse}
C^{-1}\frac{e^{-\sqrt{\nu}|x|/\sqrt{\mu}}}{|x|^{n-2}}\le \mathcal{G}_{\mu ,\nu}(x) \le C\frac{e^{-\sqrt{\nu}|x|/(2\sqrt{\mu})}}{|x|^{n-2}},\quad x\in \mathbb{R}^n,
\end{equation}
for some constant $C=C(n,\mu,\nu)>1$.
\end{lemma} 

The main result of this section is the following theorem.

\begin{theorem}\label{fundamental-solution}
Let $(a,b)\in \mathcal{D}(\lambda, \kappa)$. Then there exists a unique function $G_{a,b}$ satisfying $G_{a,b}(\cdot ,\xi )\in C(\mathbb{R}^n\setminus\{\xi\})$, $\xi \in \mathbb{R}^n$, $G_{a,b}(x, \cdot )\in C(\mathbb{R}^n\setminus\{x\})$, $x \in \mathbb{R}^n$, and
\\
(i) $L_{a,b} G_{a,b}(\cdot ,\xi )=0$ in $\mathscr{D}'(\mathbb{R}^n\setminus\{\xi\})$, $\xi \in \mathbb{R}^n$,
\\
(ii) for any $f\in C_0^\infty (\mathbb{R}^n )$, 
\[
u(x)=\int_{\mathbb{R}^n}G_{a,b}(x,\xi)f(\xi )d\xi
\]
belongs to $H^2(\mathbb{R}^n)$ and it is the unique solution of $L_{a,b} u=f$,
\\
(iii) there exist two constants $c=c(n,\lambda)>1$ and $C=C(n,\lambda,\kappa)>1$  so that
\bean \label{maineqq}
C^{-1}\frac{ e^{-2\sqrt{c \kappa}|x-\xi|}}{|x-\xi|^{n-2}}\le G_{a,b}(x,\xi) 
\le C \frac{e^{-\frac{|x-\xi|}{\sqrt{c\kappa}}}}{|x-\xi |^{n-2}}.
\eean
\end{theorem}

\begin{proof}
Pick $s\ge 1$ arbitrary and let $f\in C_0^\infty (\mathbb{R}^n )$. Applying H\"older's inequality, we find
\[
\int_{\mathbb{R}^n}k_{a,b}(t,x,\xi)|f(\xi)|d\xi \le \|k_{a,b}(t,x,\cdot )\|_{L^s(\mathbb{R}^n)}\|f\|_{L^{s'}(\mathbb{R}^n)},
\]
where $s'$ is the conjugate exponent of $s$.

But, according to \eqref{ugb},
\[
\|k_{a,b}(t,x,\cdot )\|_{L^s(\mathbb{R}^n)}^s\le \left(\frac{c}{t^{n/2}}\right)^s\int_{\mathbb{R}^n}e^{-\frac{s|x-\xi |^2}{ct}}d\xi .
\]
Next, making the change of variable $\xi =(\sqrt{ct/s})\eta +x$, we get 
\[
\|k_{a,b}(t,x,\cdot )\|_{L^s(\mathbb{R}^n)}^s\le \left(\frac{c}{t^{n/2}}\right)^s\left(\frac{ct}{s}\right)^{n/2}\int_{\mathbb{R}^n}e^{-|\eta |^2}d\eta  .
\]
Hence
\[
\|k_{a,b}(t,x,\cdot )\|_{L^s(\mathbb{R}^n)}\le t^{n(1/s-1)/2}C_s,
\] 
with 
\[
C_s=c\left(\frac{c}{s}\right)^{n/2}\left(\int_{\mathbb{R}^n}e^{-|\eta |^2}d\eta\right)^{1/s}.
\]
We get, by choosing $1\le s< \frac{n}{n-2}<\tilde{s}$,
\begin{align*}
\int_0^{+\infty}&\int_{\mathbb{R}^n}k_{a,b}(t,x,\xi)|f(\xi)|d\xi dt
\\
&=\int_0^1\int_{\mathbb{R}^n}k_{a,b}(t,x,\xi)|f(\xi)|d\xi dt+\int_1^{+\infty}\int_{\mathbb{R}^n}k_{a,b}(t,x,\xi)|f(\xi)|d\xi dt
\\
&\le C_s\|f\|_{L^{s'}(\mathbb{R}^n)}\int_0^1t^{\frac{n}{2}(1/s-1)}dt+C_{\tilde{s}}\|f\|_{L^{\tilde{s}'}(\mathbb{R}^n)}\int_1^{+\infty}t^{\frac{n}{2}(1/\tilde{s}-1)}dt.
\end{align*}
In light of Fubini's theorem we obtain
\begin{equation}\label{permutation}
\int_0^{+\infty}\int_{\mathbb{R}^n}k_{a,b}(t,x,\xi)f(\xi)d\xi dt= \int_{\mathbb{R}^n}\left(\int_0^{+\infty}k_{a,b}(t,x,\xi)dt\right) f(\xi)d\xi .
\end{equation}

Define $G_{a,b}$ as follows
\[
G_{a,b}(x,\xi )=\int_0^{+\infty}k_{a,b}(t,x,\xi)dt.
\]
Then \eqref{permutation} takes the form
\begin{equation}\label{identity1}
\int_0^{+\infty}\int_{\mathbb{R}^n}k_{a,b}(t,x,\xi)f(\xi)d\xi dt=\int_{\mathbb{R}^n}G_{a,b}(x,\xi) f(\xi)d\xi .
\end{equation}
Noting that $A_{a,b}$ is invertible,  we obtain
\begin{align*}
-A_{a,b}^{-1}f(x)&=\left(\int_0^{+\infty}e^{tA_{a,b}}fdt\right)(x)
\\
&=\int_0^{+\infty}\int_{\mathbb{R}^n}k_{a,b}(t,x,\xi)f(\xi)d\xi dt,\quad  x\in \mathbb{R}^n.
\end{align*}
This and \eqref{identity1} entail
\[
-A_{a,b}^{-1}f(x)=\int_{\mathbb{R}^n}G_{a,b}(x,\xi) f(\xi)d\xi, \quad x\in \mathbb{R}^n.
\]
In other words, $u$ defined by 
\[
u(x)=\int_{\mathbb{R}^n}G_{a,b}(x,\xi) f(\xi)d\xi,\quad x\in \mathbb{R}^n,
\]
belongs to $H^2(\mathbb{R}^n)$ and satisfies $L_{a,b}u=f$.

Since, for $x\ne \xi$,
\[
\int_0^{+\infty}\frac{1}{t^{n/2}}e^{-\frac{|x-\xi|^2}{ct}}dt=\left(c^{n/2-1}\int_0^{+\infty}\tau^{n/2-2}e^{-\tau}d\tau\right)\frac{1}{|x-\xi|^{n-2}},
\]
we get in light of \eqref{holder1}
\[
|G_{a,b}(x+h,\xi )-G_{a,b}(x,\xi )|\le \frac{C}{|x-\xi |^{n-2+\eta}}|h|^\eta,\quad x\ne \xi,\; |h|\le |x-\xi|,
\]
where $C=C(n,\lambda,\kappa)$ is a constant. In particular, $G_{a,b}(\cdot ,\xi )\in C(\mathbb{R}^n\setminus\{\xi\})$. Similarly, using \eqref{holder2} instead of \eqref{holder1}, we obtain $G_{a,b}(x,\cdot )\in C(\mathbb{R}^n\setminus\{x\})$. More specifically we have
\begin{equation}\label{holder3}
|G_{a,b}(x,\xi +h)-G_{a,b}(x,\xi )|\le \frac{C}{|x-\xi |^{n-2+\eta}}|h|^\eta,\quad x\ne \xi,\; |h|\le |x-\xi|.
\end{equation}

Let $\xi \in \mathbb{R}^n$ and $\omega \Subset \mathbb{R}^n\setminus\{\xi\}$, and pick $g\in C_0^\infty (\omega)$. Then set
\[
w_{a,b}(y)=\int_\omega G_{a,b}(x,y )g(x)dx,\quad y \in B(\xi ,\mbox{dist}(\xi ,\overline{\omega})/2).
\]
It follows from \eqref{holder3} that, for $y \in B(\xi ,\mbox{dist}(\xi ,\overline{\omega}))$ and $|h|< \mbox{dist}(y ,\overline{\omega})$, we have
\[ 
|w_{a,b}(y +h)-w_{a,b}(y )|\le \frac{C}{\mbox{dist}(y ,\overline{\omega})^{n-2+\eta}}|h|^\eta.
\]
Therefore $w_{a,b} \in C(B(\xi ,\mbox{dist}(\xi ,\overline{\omega})/2)$.

Let $\mathcal{M}(\mathbb{R}^n)$ be the space of bounded measures on $\mathbb{R}^n$. Pick a sequence $(f_k)$ of a positive functions of $C_0^\infty(\mathbb{R}^n)$ converging in $\mathcal{M}(\mathbb{R}^n)$ to $\delta_\xi$ and let $u_k=-A_{a,b}^{-1}f_k$. In that case, according to Fubini's theorem, we have
\begin{align*}
\int_\omega u_k(x)g(x)dx&=\int_\omega\int_{\mathbb{R}^n}G_{a,b}(x,y)g(x) f_k(y )dydx
\\
&=\int_{\mathbb{R}^n}w_{a,b}(y )f_k(y )dy \longrightarrow w_{a,b}(\xi )=\int_\omega G_{a,b}(x,\xi )g(x)dx,
\end{align*}
where we used that  $\mbox{supp}f_k\subset B(\xi ,\mbox{dist}(\xi ,\overline{\omega})/2)$, provided that $k$ is sufficiently large. That is we proved that $u_k$ converges to $G_{a,b}(\cdot ,\xi)$ weakly in $L^2_{\mathrm loc}(\mathbb{R}^n\setminus\{\xi\})$  (think to the fact that $C_0^\infty (\omega)$ is dense in $L^2(\omega)$).

Now, as $L_{a,b} u_k=f_k$, we find $L_{a,b}G_{a,b}(\cdot ,\xi )=0$ in $\mathbb{R}^n\setminus\{\xi\}$ in the distributional sense.

The uniqueness  of $G_{a,b}$ follows from that of $u$ and, as $\kappa^{-1}\le b\le \kappa$, we deduce from Lemma \ref{comparaison-lemma} that
\[
E_{a, \kappa}(x, t, \xi, 0)\le E_{a,b}(x, t, \xi, 0)\le E_{a, \kappa^{-1}}(x, t, \xi, 0).
\]
But  a simple  change of variable shows that
\begin{equation} \label{ee65.1}
E_{a, \kappa^{-1}}(x, t, \xi, 0) = e^{-\kappa^{-1}t} E_{a, 0}(x, t, \xi, 0)
\end{equation}
and
\begin{equation} \label{ee65.2}
E_{a, \kappa}(x, t, \xi, 0) = e^{-\kappa t} E_{a, 0}(x, t, \xi, 0).
\end{equation}

Therefore, from  Theorem \ref{theorem-FS} and identity \eqref{ee}, there exists
a constant $c=c(n,\lambda)>1$   so that
\begin{align*}
e^{-\kappa t}\frac{(\pi c^{-1})^{n/2+1}}{\pi} E_{c^{-1}/4, 0}(x, t, \xi, 0)  \le E_{a,b}&(x, t, \xi, 0)
\\
&\le e^{-\kappa^{-1} t} \frac{(\pi c)^{n/2+1}}{\pi} 
E_{c/4, 0}(x, t, \xi, 0),
\end{align*}
which, combined  with identities \eqref{ee65.1}  and \eqref{ee65.2}, gives
\begin{align*}
 \frac{(\pi c^{-1})^{n/2+1}}{\pi}E_{c^{-1}/4, \kappa}(x, t, \xi, 0)  \le E_{a,b}(x, t, \xi, &0)
 \\
 &\le \frac{(\pi c)^{n/2+1}}{\pi}E_{c/4, \kappa^{-1}}(x, t, \xi, 0).
\end{align*}
From the uniqueness of $G_{a,b}$, we obtain by integrating over $(0, +\infty)$, with respect to $t$, each member of the above  inequalities  
\[
 \frac{(\pi c^{-1})^{n/2+1}}{\pi}G_{c^{-1}/4, \kappa}(x,\xi ) \le G_{a,b}(x,\xi ) \le \frac{(\pi c)^{n/2+1}}{\pi}G_{c/4, \kappa^{-1}}(x,\xi ).
\] 
These two-sided inequalities together with \eqref{tse} yield in a straightforward manner \eqref{maineqq}. 
\end{proof}

The function $G_{a,b}$ given by the previous theorem is usually called a fundamental solution of the operator 
$L_{a,b}$.

\subsection{Regularity of fundamental solutions}\label{sub2.2}

Let $\xi \in \mathbb{R}^n$ and $\mathcal{O}\Subset \mathcal{O}'\Subset \mathbb{R}^n\setminus\{\xi\}$ with $\mathcal{O}'$ of class $C^{1,1}$. As $G_{a,b}(\cdot ,\xi )\in C(\partial \mathcal{O}')$, we get from \cite[Theorem 6.18, page 106]{GT} (interior H\"older regularity) that $G_{a,b}(\cdot ,\xi )$  belongs to $C^{2,\alpha}(\overline{\mathcal{O}})$.

\begin{proposition}\label{propositionlr1}
There exist $C=C(n,\lambda ,\kappa ,\alpha)$ and $\nu =\nu (\alpha)>2$ so that, for any $\xi\in \mathbb{R}^n$ and $\mathcal{O}\Subset \mathbb{R}^n\setminus\{\xi\}$, we have
\begin{equation}\label{lr1}
\|G_{a,b}(\cdot ,\xi )\|_{C^{2,\alpha}(\overline{\mathcal{O}})}\le C\Lambda (\mathbf{d}+\varrho)^\nu \max\left(\varrho^{-(2+\alpha)},1\right)\varrho^{-n+2}.
\end{equation}
Here $\varrho = \mbox{dist}\left(\xi,\overline{\mathcal{O}}\right)$, $\mathbf{d}=\mbox{diam}(\mathcal{O})$ and
\[
\Lambda (h)=[1+2h+2h^2+h^3]\lambda,\quad h>0.
\]
\end{proposition}
The proof of this proposition is based the following lemma consisting in an adaptation of the usual interior Schauder estimates. The proof of this technical lemma will be given in Appendix A.

\begin{lemma}\label{lemmaA1}
There exists two constants $C=C(n,\alpha)$ and $\nu =\nu (\alpha )>1$ with the property that, for any bounded subset $\mathcal{Q}$ of $\mathbb{R}^n$, $\delta >0$ so that $\mathcal{Q}_\delta =\{x\in \mathcal{Q};\; \mbox{dist}(x,\partial \mathcal{Q})>\delta\}\ne \emptyset$, $w\in C^{2,\alpha}(\mathcal{Q})\cap C\left(\overline{\mathcal{Q}}\right)$ satisfying $L_{a,b}w=0$ in $\mathcal{Q}$ and $\mathcal{Q}'\subset \mathcal{Q}_\delta$, we have
\begin{equation}\label{A6}
\|w\|_{C^{2,\alpha}\left(\overline{\mathcal{Q}'}\right)} \le C\max\left(\delta^{-(2+\alpha)},1\right)\Lambda (\mathbf{d})^\nu \|w\|_{C\left(\overline{\mathcal{Q}}\right)},
\end{equation}
where $\Lambda$ is as in Proposition \ref{propositionlr1} and $\mathbf{d}=\mbox{diam}(\mathcal{Q})$.
\end{lemma}

\begin{proof}[Proof of Proposition \ref{propositionlr1}]
We get, by applying Lemma \ref{lemmaA1} with $\mathcal{Q}'=\mathcal{O}$, $\delta =\varrho/2$ and $\mathcal{Q}=\left\{x\in \mathbb{R}^n;\; \mbox{dist}\left(x, \overline{\mathcal{O}}\right)<\varrho/2\right\}$, 
\[
\|G_{a,b}(\cdot ,\xi )\|_{C^{2,\alpha}(\overline{\mathcal{O}})}\le C\Lambda (\mathbf{d}+\varrho)^\nu \max\left(\delta^{-(2+\alpha)},1\right)\|G_{a,b}(\cdot ,\xi )\|_{C\left(\overline{\mathcal{Q}}\right)}.
\]
This and \eqref{maineqq} yield
\begin{equation}\label{A7.0}
\|G_{a,b}(\cdot ,\xi )\|_{C^{2,\alpha}(\overline{\mathcal{O}})}\le C\Lambda (\mathbf{d}+\varrho)^\nu \max\left(\delta^{-(2+\alpha)},1\right)\varrho^{-n+2}e^{-\varrho/\sqrt{c\kappa}},
\end{equation}
with $C=C(n,\lambda ,\kappa ,\alpha)$ and $c=c(n,\lambda)$. It is then clear that \eqref{A7.0} implies \eqref{lr1}.
\end{proof}

The preceding proposition together with Lemma \ref{lemmaB2} enable us to state the following corollary.

\begin{corollary}\label{corollaryq0}
There exist $C=C(n,\lambda ,\kappa ,\alpha,\theta)$ and $\nu =\nu (\alpha)>1$ so that, for any $\xi\in \mathbb{R}^n$ and $\mathcal{O}\Subset \mathbb{R}^n\setminus\{\xi\}$, we have
\begin{align}
&\|G_{a,b}(\cdot ,\xi )\|_{H^{2+\theta}(\mathcal{O})}\label{lr1.0}
\\
&\qquad \le C\Lambda (\mathbf{d}+\varrho)^\nu\max\left(\mathbf{d}^{n/2},\mathbf{d}^{n/2+\alpha -\theta}\right) \max\left(\varrho^{-(2+\alpha)},1\right)\varrho^{-n+2},\nonumber
\end{align}
where $\varrho = \mathrm{dist}\left(\xi,\overline{\mathcal{O}}\right)$, $\mathbf{d}=\mathrm{diam}(\mathcal{O})$.
\end{corollary}

\begin{corollary}\label{corollaryq1}
There exist $C=C(n,\lambda ,\kappa ,\alpha)$ and $c=c(n,\lambda,\kappa,\alpha )$  so that, for any $\xi_1,\xi_2\in \mathbb{R}^n$ and $\mathcal{O}\Subset \mathbb{R}^n \setminus\{\xi_1,\xi_2\}$, we have
\begin{equation}\label{A11.0}
\left\|\frac{G_{a,b}(\cdot ,\xi_2)}{G_{a,b}(\cdot ,\xi_1)}\right\|_{C^{2,\alpha}(\overline{\mathcal{O}})}\le Ce^{c(\mathbf{d}+\varrho_+)}\left(1+ \max\left(\varrho_-^{-(2+\alpha)},1\right)\varrho_-^{-n+2}\right)^4,
\end{equation}
where $\varrho_-=\min \left(\mathrm{dist}\left(\xi_1,\mathcal{O}\right),\mathrm{dist}\left(\xi_2,\mathcal{O}\right)\right)$ and $\varrho_+=\max \left(\mathrm{dist}\left(\xi_1,\mathcal{O}\right),\mathrm{dist}\left(\xi_2,\mathcal{O}\right)\right)$.
\end{corollary}

\begin{proof}
In this proof $C=C(n,\lambda ,\kappa ,\alpha)$, $c=c(n,\lambda,\kappa,\alpha)$ and $\nu =\nu (\alpha)>2$ are generic constants.

From Proposition \ref{propositionlr1}, we have
\begin{equation}\label{A9.0}
\|G_{a,b}(\cdot ,\xi_j )\|_{C^{2,\alpha}(\overline{\mathcal{O}})}\le C\Lambda (\mathbf{d}+\varrho_+)^\nu \max\left(\varrho_-^{-(2+\alpha)},1\right)\varrho_-^{-n+2},\quad j=1,2.
\end{equation}
Let $C_0\ge 1$ end $c_0\ge 1$ be the constants in \eqref{maineqq} and fix $0<\delta_0\le 1$. Then the first inequality in \eqref{maineqq} gives
\[
\frac{1}{G_{a,b}(\cdot ,\xi_1)}\le C_0\left(\mathbf{d}+\varrho_+\right)^{n-2}e^{2\sqrt{c_0\kappa}(\mathbf{d}+\varrho_+)}.
\]
This inequality together with Lemma \ref{lemmaB1} in Appendix \ref{appendixA} yield
\begin{equation}\label{A10.0}
\left\|\frac{1}{G_{a,b}(\cdot ,\xi_1)}\right\|_{C^{2,\alpha}(\overline{\mathcal{O}})}\le Ce^{c(\mathbf{d}+\varrho_+)}\left(1+\|G_{a,b}(\cdot ,\xi_1 )\|_{C^{2,\alpha}(\overline{\mathcal{O}})}\right)^3.
\end{equation}
Then in light of \eqref{A9.0} and \eqref{A10.0}, we get in a straightforward manner 
\[
\left\|\frac{G_{a,b}(\cdot ,\xi_2)}{G_{a,b}(\cdot ,\xi_1)}\right\|_{C^{2,\alpha}(\overline{\mathcal{O}})}\le Ce^{c(\mathbf{d}+\varrho_+)}\left(1+(1+\mathbf{d})^\nu \max\left(\varrho_-^{-(2+\alpha)},1\right)\varrho_-^{-n+2}\right)^4, 
\]
and hence
\[
\left\|\frac{G_{a,b}(\cdot ,\xi_2)}{G_{a,b}(\cdot ,\xi_1)}\right\|_{C^{2,\alpha}(\overline{\mathcal{O}})}\le Ce^{c(\mathbf{d}+\varrho_+)}\left(1+ \max\left(\varrho_-^{-(2+\alpha)},1\right)\varrho_-^{-n+2}\right)^4. 
\]
This is the expected inequality.
\end{proof}

This corollary combined with Lemma \ref{lemmaB2} yields the following result.
\begin{corollary}\label{corollaryq2}
There exist $C=C(n,\lambda ,\kappa ,\alpha, \theta)$ and $c=c(n,\lambda,\kappa,\alpha,\theta )$  so that, for any $\xi_1,\xi_2\in \mathbb{R}^n$ and $\mathcal{O}\Subset \mathbb{R}^n \setminus\{\xi_1,\xi_2\}$, we have
\begin{equation}\label{A11.0}
\left\|\frac{G_{a,b}(\cdot ,\xi_2)}{G_{a,b}(\cdot ,\xi_1)}\right\|_{H^{2+\theta}(\mathcal{O})}\le Ce^{c(\mathbf{d}+\varrho_+)}\left(1+ \max\left(\varrho_-^{-(2+\alpha)},1\right)\varrho_-^{-n+2}\right)^4.
\end{equation}
Here $\varrho_\pm$ is the same as in Corollary \ref{corollaryq1}.
\end{corollary}

\subsection{Gradient estimate of the quotient of two fundamental solutions}

The following result uses  the singularity of the Green function near the location  of the  point source.  

\begin{lemma}\label{lemmaq1} 
There exist $x^\ast\in  B(\xi_2,|\xi_1-\xi_2|/2)\setminus\{\xi_2\}$, $C=(n,\lambda,\kappa, |\xi_1-\xi_2|)>0$ and $\rho=\rho(n,\lambda,\kappa, |\xi_1-\xi_2|)>0$  so that $\overline{B}(x^\ast,\rho)\subset B(\xi_2,|\xi_1-\xi_2|/2)\setminus\{\xi_2\}$ and 
\bea
C\le \left\| \nabla \left( \frac{G_{a,b}(\cdot ,\xi_2)}{G_{a,b}(\cdot ,\xi_1)}\right)\right\|_{L^2(B(x^\ast,\rho))}.
\eea
\end{lemma}

\begin{proof}
We set for notational convenience $w=G_{a,b}(\cdot ,\xi_2)/G_{a,b}(\cdot ,\xi_1)$. In light of Theorem \ref{fundamental-solution}, we obtain by straightforward computations the following two-sided inequality
\begin{equation}\label{ge1}
\frac{C^{-1}}{|x-\xi_2|^{n-2}}\le w(x)\le \frac{C}{|x-\xi_2|^{n-2}},\quad x\in B(\xi_2,|\xi_1-\xi_2|/2)\setminus\{\xi_2\}.
\end{equation}
Here and until the end of this proof $C=C(n,\lambda,\kappa, |\xi_1-\xi_2|)$ is a generic constant.

Set $\tilde{t}=|\xi_1-\xi_2|/4$ and define
\bea
\varphi(t,\theta) = w(\xi_2+t \theta),\quad (t,\theta )\in (0,\tilde{t}]\times \mathbb{S}^{n-1}. 
\eea
According to Corollary \ref{corollaryq1}, $\varphi \in C_{\mathrm loc}^{2,\alpha}((0,\tilde{t}]\times \mathbb{S}^{n-1})$ and hence
\bea
\varphi (\tilde{t},\theta ) - \varphi(t,\theta ) = \int_t^{\tilde{t}} \nabla w(\xi_2+s \theta)\cdot \theta ds,
\eea
which in turn gives
\begin{align*}
|\varphi (\tilde{t},\theta) - \varphi(t,\theta)|^2 &\leq (\tilde{t}-t)\int_t^{\tilde{t}} \left|\nabla w(\xi_2+s\theta)\right|^2ds 
\\
&\le \tilde{t}\int_t^{\tilde{t}} \left|\nabla w(\xi_2+s\theta)\right|^2ds
\\
&\le \tilde{t}\int_t^{\tilde{t}} \frac{s^{n-1}}{t^{n-1}}\left|\nabla w(\xi_2+s\theta)\right|^2ds, \quad (t,\theta )\in (0,\tilde{t}]\times \mathbb{S}^{n-1}.
\end{align*}
Whence, where $t\in (0,\tilde{t}]$,
\begin{equation} \label{r1}
t^{n-1}\int_{\mathbb{S}^{n-1}}|\varphi(\tilde{t},\theta) - \varphi(t,\theta)|^2d\theta \le \tilde{t}  
\int_{\mathscr C_t} \left|\nabla w(x)\right|^2dx.
\end{equation}
Here
\bea
\mathscr C_t = \left\{ x\in \mathbb R^n; \;  t<|x-\xi_2| <\tilde{t}\right\}.
\eea
On the other hand inequalities \eqref{ge1} imply, where $(t,\theta )\in (0,\tilde{t}]\times \mathbb{S}^{n-1}$,
\begin{equation*}\label{ge2}
\frac{C^{-1}}{t^{n-2}}
\le \varphi(t, \theta) \le \frac{C} {t^{n-2}}.
\end{equation*}
Let us then choose $t_0\le \tilde{t}$ sufficiently small in such a way that 
\[
 \frac{C^{-1}}{t^{n-2}}- \frac{C} {\tilde{t}^{n-2}} > 0,\quad t\in (0,t_0].
\] 
Therefore, for $ (t,\theta)\in (0,t_0]\times \mathbb{S}^{n-1}$,  we have
\begin{equation}\label{r2}
\left( \frac{C^{-1}}{t^{n-2}}- \frac{C} {\tilde{t}^{n-2}}\right)^2\le 
 |\varphi(\tilde{t},\theta) - \varphi(t,\theta)|^2 .
\end{equation}
We then obtain by combining inequalities  \eqref{r1} and \eqref{r2} 
\[
|\mathbb{S}^{n-1}|\left( \frac{C^{-1}}{t^{n-2}}- \frac{C} {\tilde{t}^{n-2}}\right)^2
 \le \tilde{t}\int_{\mathscr C_t} \left|\nabla w(x)\right|^2dx, \quad t\in (0,t_0].
\]
We have in particular
\bea
C\le \int_{\mathscr C_{t_0}} \left|\nabla w(x)\right|^2dx.
\eea

Let $\rho = t_0/4$. Then it is straightforward to check that, for any $x\in \overline{\mathscr{C}_{t_0}}$,
\[
\overline{B}(x,\rho)\subset \{ y\in \mathbb{R}^n;\; 3t_0/4\le |y-\xi_2|\le 5\tilde{t}/4\}\subset B(\xi_2,|\xi_1-\xi_2|/2)\setminus\{\xi_2\}.
\]
Since $\overline {\mathscr C_{t_0}}$ is compact, we find a positive integer $N=N(\lambda,\kappa,|\xi_1-\xi_2|)$ and $x_j \in \overline {\mathscr C_{t_0}}$, $j=1,\cdots, N$,  so that 
\bea
\overline {\mathscr C_{t_0}} \subset \bigcup_{j=1}^N B(x_j, \rho).
\eea
Hence 
\bea
C\leq  
\int_{\cup_{j=1}^N B(x_j, \rho)} \left|\nabla w(x)\right|^2dx. 
\eea
Pick then $x^\ast \in \{x_j, \;  1\le j\le N\}$ in such a way that 
\[
\int_{B(x^\ast, \rho)} \left|\nabla w(x)\right|^2dx = \max_{1\le j\le N} \int_{B(x_j, \rho)} \left|\nabla w(x)\right|^2dx.
\]
Therefore 
\bea
C\le  \int_{ B(x^\ast, \rho)} \left|\nabla w(x)\right|^2dx. 
\eea
This finishes the proof.
\end{proof}

\section{Uniform lower bound for the gradient}\label{section3}

Let $\mathcal{O}$ be a Lipschitz bounded domain of $\mathbb{R}^n$ and $\sigma \in C^{0,1}(\overline{\mathcal O})$ satisfying 
\bean \label{sigmabounds1}
\varkappa^{-1} \le \sigma \quad \mbox{and}\quad \|\sigma\|_{C^{0,1}(\overline{\mathcal O})} \le \varkappa,
\eean
for some fixed constant $\varkappa>1$.

We  prove in this section a polynomial lower bound  of the local $L^2$-norm of the
gradient of solutions of 
\[
L_\sigma u =\mbox{div}(\sigma \nabla u)= 0\quad \mbox{in}\; \mathcal{O}. 
\]
In a first step we establish, via a three-ball inequality for the gradient, a uniform lower bound of negative exponential type. We use then in a second  step an argument based on  the so-called frequency function in order to improve this lower bound.

\subsection{Preliminary lower bound}

We need hereafter the following three-ball inequality for the gradient.

 \begin{theorem}\label{theoremlb1}
Let $0<k<\ell<m$ be real. There exist two constants $C=C(n,\varkappa,k,\ell,m)>0$ and $0<\gamma=\gamma(n,\varkappa,k,\ell,m) <1$  so that, for any  $v\in H^1(\mathcal{O})$ satisfying $L_\sigma v=0$, $y\in \mathcal{O}$ and $0<r< \mbox{dist}(y,\partial \mathcal{O})/m$, we have
\[
C\|\nabla v\|_{L^2(B(y,\ell r))}\le \|\nabla v\|_{L^2(B(y,kr))}^\gamma \|\nabla v\|_{L^2(B(y,m r))}^{1-\gamma}.
\]
\end{theorem}
A proof of this theorem can be found  in \cite{BC} or  \cite{Ch}.

Define the geometric distance $d_g^D$ on the bounded domain $D$ of $\mathbb{R}^n$ by
\[
d_g^D(x,y)=\inf\left \{ \ell (\psi ) ;\; \psi :[0,1]\rightarrow D \; \mbox{Lipschitz path joining}\; x \; \mbox{to}\; y\right\},
\] 
where
\[
\ell (\psi )= \int_0^1|\dot{\psi}(t)|dt
\]
is the length of $\psi$. 

Note that according to Rademacher's theorem any Lipschitz continuous function $\psi :[0,1]\rightarrow D$ is almost everywhere differentiable with $|\dot{\psi}(t)|\le k$ a.e. $t\in [0,1]$, where $k$ is the Lipschitz constant of $\psi$. 

\begin{lemma}\label{Glemma} 
Let $D$ be a bounded Lipschitz domain of $\mathbb{R}^n$. Then $d_g^D\in L^\infty (D \times D )$ and there exists a constant $\mathfrak{c}_D>0$ so that
\bean \label{distance}
|x-y|\le d^D_g(x,y)\le \mathfrak{c}_D|x-y|, \quad  x, y \in D.
\eean
\end{lemma}

We refer to \cite[Lemma A3]{TW} for a proof.

In this subsection we use the following notations
\[
\mathcal{O}^\delta =\{ x\in \mathcal{O};\; \mbox{dist}(x,\partial \mathcal O)>\delta\}
\]
and
\[
\chi (\mathcal{O})=\sup\{ \delta >0;\; \mathcal{O}^\delta\ne\emptyset\}.
\]
Define 
\begin{align} 
\mathscr{S}(\mathcal{O}, x_0,M,\eta ,\delta)=
&\left\{ u\in H^1(\mathcal{O}); \; L_\sigma u=0\; \textrm{in}\; \mathcal{O}, \right.\label{S}
 \\
 &\hskip 2cm \left. \|\nabla u\|_{L^2(\mathcal{O})}\le M,\;  \|\nabla u\|_{L^2(B(x_0,\delta ))}\ge \eta \right\},\nonumber
\end{align}
with $\delta \in (0, \chi (\mathcal{O})/3)$, $x_0\in \mathcal{O}^{3\delta}$, $\eta>0$ and $M\ge 1$ satisfying $\eta <M$.

\begin{lemma}\label{lemmalb1}
There exist two constants $c=c(n,\varkappa)\ge 1$ and $0<\gamma=\gamma(n,\varkappa)<1$  so that, for any $u\in\mathscr{S}(\mathcal{O}, x_0, M,\eta ,\delta)$ and $x\in  \mathcal{O}^{3\delta}$, we have
\begin{equation}\label{2.1.1}
e^{-[\ln(cM/\eta)/\gamma] e^{[2n|\ln \gamma|] \mathfrak{c}|x-x_0|/\delta }}\le \|\nabla u\|_{L^2(B(x,\delta ))},
\end{equation}
with $\mathfrak{c}=\mathfrak{c}_{\mathcal{O}}$ is as in Lemma \ref{Glemma}.
\end{lemma}

\begin{proof}
Pick $u\in \mathscr{S}(\mathcal{O},x_0, M,\eta ,\delta)$. Let $x\in \mathcal{O}^{3\delta}$ and $\psi :[0,1]\rightarrow \mathcal{O}$ be a Lipschitz path joining $x=\psi(0)$ to $x_0=\psi(1)$, so that $\ell (\psi)\le  2d_g (x_0,x)$. Here and henceforth, for simplicity convenience, we use $d_g(x_0,x)$  instead of $d_g^{\mathcal{O}}(x_0,x)$.

Let $t_0=0$ and $t_{k+1}=\inf \{t\in [t_k,1];\; \psi (t)\not\in B(\psi (t_k),\delta )\}$, $k\geq 0$. We claim that there exists an integer $N\ge 1$ verifying  $\psi (1)\in B(\psi(t_N),\delta )$. If not, we would have $\psi (1)\not\in B(\psi (t_k),\delta )$ for any $k\ge 0$. As the sequence $(t_k)$ is non decreasing and bounded from above by $1$, it converges to $\hat{t}\le 1$. In particular, there exists an integer $k_0\geq 1$ so that $\psi (t_k)\in B\left(\psi (\hat{t}),\delta/2\right)$, $k\ge k_0$. But this contradicts the fact that $\left|\psi (t_{k+1})-\psi (t_k)\right| \ge \delta$, $k\ge 0$.

Let us check that $N\le N_0$, where $N_0=N_0(n, |x-x_0|, \mathfrak{c}, \delta)$. Pick $1\le j\le n$ so that 
\[
\max_{1\leq i\leq n} \left|\psi _i(t_{k+1})-\psi _i(t_k)\right| =\left|\psi _j(t_{k+1})-\psi _j(t_k)\right|,
\]
where $\psi_i$ is the $i$th component of $\psi$.
Then
\[
\delta \le n\left|\psi _j (t_{k+1})-\psi _j(t_k)\right|=n\left| \int_{t_k}^{t_{k+1}}\dot{\psi}_j(t)dt\right|\le  n\int_{t_k}^{t_{k+1}}|\dot{\psi}(t)|dt  .
\]
Consequently, where $t_{N+1}=1$, 
\[
(N+1)\delta \le n\sum_{k=0}^N\int_{t_k}^{t_{k+1}}|\dot{\psi}(t)|dt=n\ell (\psi)\le 2nd_g (x_0,x)\le 2n\mathfrak{c}|x-x_0|.
\]
Therefore
\[
N\le N_0=\left[ \frac{2n\mathfrak{c}|x-x_0|}{\delta}\right].
\]

Let $y_0=x$ and $y_k=\psi (t_k)$, $1\le k\le N$.  If $|z-y_{k+1}|<\delta$, then $|z-y_k|\le |z-y_{k+1}|+|y_{k+1}-y_k|<2\delta$. In other words $B(y_{k+1},\delta )\subset B(y_k,2\delta)$. We get from Theorem \ref{theoremlb1}
\begin{equation}\label{est1}
\|\nabla u\|_{L^2(B(y_j,2\delta ))}\leq C\|\nabla u\|_{L^2(B(y_j,3\delta ))}^{1-\gamma}\|\nabla u\|_{L^2(B(y_j,\delta ))}^\gamma,\quad 0\le j\le N,
\end{equation}
for some constants $C=C(n,\varkappa)>0$ and $0<\gamma=\gamma(n,\varkappa) <1$.

Set $I_j=\|\nabla u\|_{L^2(B(y_j,\delta ))}$, $0\le j\le N$ and $I_{N+1}=\|\nabla u\|_{L^2(B(x_0,\delta ))}$. Since $B(y_{j+1},\delta )\subset B(y_j,2\delta )$, $1\le j\le N-1$, estimate \eqref{est1} implies
\begin{equation}\label{est2}
I_{j+1}\le C M^{1-\gamma}I_j^\gamma,\;\; 0\le j\le N.
\end{equation}
Let $C_1=C^{1+\gamma+\ldots +\gamma^{N+1}}$ and $\beta =\gamma^{N+1}$. Then, by a simple induction argument, estimate \eqref{est2} yields
\begin{equation}\label{est3}
I_{N+1}\leq C_1M^{1-\beta}I_0^\beta .
\end{equation}
Without loss of generality, we assume in the sequel that $C\ge 1$ in \eqref{est2}.
Using that $N\leq N_0$, we have 
\begin{align*}
&\beta \ge \beta _0=\gamma^{N_0+1},
\\
&C_1\le C^{\frac{1}{1-\gamma}},
\\
&\left(\frac{I_0}{M}\right)^\beta\le \left(\frac{I_0}{M}\right)^{\beta_0}.
\end{align*}
These estimates in \eqref{est3} give
\[
\frac{I_{N+1}}{M}\leq C^{\frac{1}{1-\gamma}}\left(\frac{I_0}{M}\right)^{\gamma^{N_0+1}},
\]
from which we deduce that
\[
\|\nabla u\|_{L^2(B(x_0,\delta ))}\le C^{\frac{1}{1-\gamma}}M^{1-\gamma^{N_0+1}}\|\nabla u\|_{L^2(B(x,\delta ))}^{\gamma^{N_0+1}}.
\]
But $M\ge 1$. Whence
\[
\eta \le \|\nabla u\|_{L^2(B(x_0,\delta ))}\le C^{\frac{1}{1-\gamma}}M\|\nabla u\|_{L^2(B(x,\delta ))}^{\gamma^{N_0+1}}.
\]
The expected inequality follows readily from this last estimate.
\end{proof}

\subsection{An estimate for the frequency function}

Some tools in the present section are borrowed from \cite{GL1, GL2, Ku}. Let $u\in H^1(\mathcal{O})$ and $\sigma \in C^{0,1}(\overline{\mathcal O})$ satisfying  the bounds
\eqref{sigmabounds1}. We recall that the usual frequency function, relative to the operator $L_\sigma $, associated to $u$ is defined by
\[
N(u)(x_0,r)= \frac{rD(u)(x_0,r)}{H(u)(x_0,r)},
\]
provided that $B(x_0,r)\Subset \mathcal{O}$, with
\begin{align*}
&D(u)(x_0,r)=\int_{B(x_0,r)}\sigma (x)|\nabla u(x)|^2dx,
\\
&H(u)(x_0,r)=\int_{\partial B(x_0,r)}\sigma (x) u^2(x)dS(x).
\end{align*}
Define also
\[
K(u)(x_0,r)= \int_{B(x_0,r)}\sigma(x)u^2(x)dx.
\]
Prior to studying the properties of the frequency function, we prove some preliminary results. Fix $u\in H^2(\mathcal{O})$ so that $L_\sigma u =0$ in $\mathcal{O}$ and, for simplicity convenience, we drop in the sequel the dependence on $u$ of $N$, $D$, $H$ and $K$.

\begin{lemma}\label{lemma-a1}
For $x_0\in \mathcal{O}^\delta$ and $0<r<\delta$, we have
\begin{align}
&\partial _rH(x_0,r)=\frac{n-1}{r}H(x_0,r)+\tilde{H}(x_0,r) +2D(x_0,r), \label{a1}
\\
&\partial _rD(x_0,r)=\frac{n-2}{r}D(x_0,r)+\frac{1}{r}\tilde{D}(x_0,r)+2\hat{H}(x_0,r). \label{a2}
\end{align}
Here 
\begin{align*}
&\tilde{H}(x_0,r)=\int_{\partial B(x_0,r)}u^2\nabla \sigma(x)\cdot \nu (x)dS(x),
\\
&\hat{H}(x_0,r)=\int_{\partial B(x_0,r)}\sigma(x)(\partial _\nu u(x))^2dS(x),
\\
&\tilde{D}(x_0,r)=\int_{B(x_0,r)}|\nabla u(x)|^2\nabla \sigma(x)\cdot (x-x_0)dx.
\end{align*}
\end{lemma}

\begin{proof}
Pick $x_0\in \mathcal{O} ^\delta$ and $0<r<\delta$. A simple change of variable yields
\[
H(x_0,r)=\int_{\partial B(0,1)}\sigma(x_0+ry)u^2(x_0+ry)r^{n-1}dS(y).
\]
Hence
\begin{align*}
\partial_rH(x_0,r)&=\frac{n-1}{r}H(x_0,r)+\int_{\partial B(0,1)}\nabla (\sigma u^2)(x_0+ry)\cdot yr^{n-1}dS(y)
\\
&= \frac{n-1}{r}H(x_0,r)+\int_{\partial B(0,1)}u^2(x_0+ry)\nabla \sigma(x_0+ry)\cdot yr^{n-1}dS(y)
\\
&\hskip 3cm +\int_{\partial B(0,1)}\sigma(x_0+ry)\nabla (u^2)(x_0+ry)\cdot yr^{n-1}dS(y)
\\
& = \frac{n-1}{r}H(x_0,r)+\int_{\partial B(x_0,r)}u^2(x)\nabla \sigma (x)\cdot \nu (x)dS(x)
\\
&\hskip 3cm + \int_{\partial B(x_0,r)}\sigma(x)\nabla (u^2)(x)\cdot \nu (x)dS(x)
\\
&=\frac{n-1}{r}H(x_0,r)+\tilde{H}(x_0,r)+\int_{\partial B(x_0,r)}\sigma(x) \nabla (u^2)(x)\cdot \nu (x)dS(x).
\end{align*}
Identity \eqref{a1} will follow if we prove
\begin{equation}\label{eq1}
2D(x_0,r)=\int_{\partial B(x_0,r)}\sigma(x) \nabla (u^2)(x)\cdot \nu (x)dS(x).
\end{equation}
To this end,  we observe that $\mbox{div}(\sigma \nabla u)=0$ implies
\[
\mathrm{div}(\sigma\nabla (u^2))=2u\mathrm{div}(\sigma \nabla u)+2\sigma |\nabla u|^2=2\sigma |\nabla u|^2.
\]
We then get by applying the divergence theorem 
\begin{align}
2D(x_0,r)&=\int_{B(x_0,r)}\mathrm{div}(\sigma(x) \nabla (u^2)(x))dx \label{a6}
\\
&=\int_{\partial B(x_0,r)}\sigma(x) \nabla (u^2)(x)\cdot \nu (x)dS(x).\nonumber
\end{align}
This proves \eqref{eq1}.

By a change of variable we have
\[
D(x_0,r)=\int_0^r\int_{\partial B(0,1)}\sigma(x_0+ty)|\nabla u(x_0+ty)|^2 t^{n-1}dS(y)dt.
\]
Hence
\begin{align*}
\partial _rD(x_0,r)&=\int_{\partial B(0,1)}\sigma(x_0+ry)|\nabla u(x_0+ry)|^2r^{n-1}dS(y)
\\
&= \int_{\partial B(x_0,r)}\sigma(x)|\nabla u(x)|^2dS(x) 
\\
&= \frac{1}{r}\int_{\partial B(x_0,r)}\sigma(x)|\nabla u(x)|^2(x-x_0)\cdot \nu (x)dS(x).
\end{align*}
An application of the divergence theorem then gives
\[
\partial _rD(x_0,r)=\frac{1}{r}\int_{B(x_0,r)}\mathrm{div}(\sigma(x)|\nabla u(x)|^2(x-x_0))dx.
\]
Therefore
\begin{align*}
\partial _rD(x_0,r)=\frac{1}{r}\int_{B(x_0,r)}&|\nabla u(x)|^2\mathrm{div}(\sigma(x)(x-x_0))dx
\\
&+\frac{1}{r}\int_{B(x_0,r)}\sigma(x)(x-x_0)\cdot \nabla (|\nabla u(x)|^2)dx
\end{align*}
implying
\begin{align}
\partial _rD(x_0,r)=\frac{n}{r}D(x_0,r)&+\frac{1}{r}\tilde{D}(x_0,r) \label{a3}
\\
&+\frac{1}{r}\int_{B(x_0,r)}\sigma(x)(x-x_0)\cdot \nabla (|\nabla u(x)|^2)dx.\nonumber
\end{align}

On the other hand,
\begin{align*}
\int_{B(x_0,r)}& \sigma(x)(x_j-x_{0,j}) \partial_j(\partial _i u(x))^2dx
\\
&=2\int_{B(x_0,r)}\sigma(x)(x_j-x_{0,j}) \partial_{ij}^2 u\partial _i u(x)dx
\\
&= -2\int_{B(x_0,r)}\partial _i\left[\partial _i u(x)\sigma(x)(x_j-x_{0,j})\right]\partial _j u(x)dx
\\ 
&\hskip 2cm +2\int_{\partial B(x_0,r)} \sigma(x)\partial _i u(x)(x_j-x_{0,j})\partial_j u(x)\nu _i(x)dS(x)
\\
&= -2\int_{B(x_0,r)} \partial _{ii}^2 u(x)\sigma(x)(x_j-x_{0,j})\partial _j u(x)dx
\\
&\hskip 1cm -2\int_{B(x_0,r)} \partial _i u(x) \partial _ju(x)\partial _i\left[\sigma(x)(x_j-x_{0,j})\right]dx
\\ 
&\hskip 2cm +2\int_{\partial B(x_0,r)}\sigma(x)\partial _i u(x)(x_j-x_{0,j})\partial_j u(x)\nu _i(x)dS(x).
\end{align*}
Thus, taking into account that $\sigma \Delta u=-\nabla \sigma \cdot \nabla u$, 
\begin{align*}
\int_{B(x_0,r)}\sigma(x)(x-x_0)\cdot \nabla (|\nabla u(x)|^2)dx&=-2\int_{B(x_0,r)} \sigma(x)|\nabla u(x)|^2dx
\\
& +2r\int_{\partial B(x_0,r)}\sigma(x)(\partial _\nu u(x))^2dS(x).
\end{align*}
This identity in \eqref{a3} yields
\[
\partial _rD(x_0,r)=\frac{n-2}{r}D(x_0,r)+\frac{1}{r}\tilde{D}(x_0,r)  +2\hat{H}(x_0,r).
\]
That is we proved \eqref{a2}.
\end{proof}

\begin{lemma}\label{lemma-a3}
We have
\[
K(x_0,r)\le re^{r\varkappa^2}H(x_0,r),\quad x_0\in \mathcal{O} ^\delta ,\; 0<r<\delta. 
\]
\end{lemma}

\begin{proof}
Taking into account that $H(x_0,r)\ge 0$ and $D(x_0,r)\ge 0$, we obtain from identity \eqref{a1}
\begin{align*}
\partial_rH(x_0,r)
&\ge \int_{\partial B(x_0,r)}\partial _\nu \sigma(x)u^2(x)dS(x)
\\
&\ge \int_{\partial B(x_0,r)}\frac{\partial _\nu \sigma(x)}{\sigma(x)}\sigma(x)u^2(x)dS(x)\ge -\varkappa^2 H(x_0,r).
\end{align*}
Consequently $r\rightarrow e^{r\varkappa^2}H(x_0,r)$ is non decreasing and then
\begin{align*}
\int_0^r H(x_0,t) dt&\le\int_0^r e^{t \varkappa^2}H(x_0,t) dt
\\
&\le \int_0^r e^{r\varkappa^2}H(x_0,r) dt\le re^{r\varkappa^2}H(x_0,r).
\end{align*}
As
\[
K(x_0,r)=\int_0^r  H(x_0,t) dt,
\]
we end up getting
\[
K(x_0,r)\le re^{r\varkappa^2}H(x_0,r).
\]
This completes the proof.
\end{proof}

Now straightforward computations yield, for $x_0\in \mathcal{O}^\delta$ and $0<r<\delta$,
\begin{equation}\label{a4}
\frac{\partial_rN(x_0,r)}{N(x_0,r)}=\frac{1}{r}+\frac{\partial_rD(x_0,r)}{D(x_0,r)}-\frac{\partial_rH(x_0,r)}{H(x_0,r)}.
\end{equation}

\begin{lemma}\label{lemma-bff}
For $x_0\in \mathcal{O}^\delta$ and $0<r<\delta$, we have
\[
N(x_0,r) \le e^{2\varkappa^2 \delta}N(x_0,\delta ).
\]

\end{lemma}

\begin{proof}
We have from formulas \eqref{a1} and \eqref{a2} and identity \eqref{a4}
\begin{align}
\frac{\partial_rN(x_0,r)}{N(x_0,r)}&=\frac{\tilde{D}(x_0,r)}{rD(x_0,r)}-\frac{\tilde{H}(x_0,r)}{H(x_0,r)}+ 2\frac{\hat{H}(x_0,r)}{D(x_0,r)}-2\frac{D(x_0,r)}{H(x_0,r)} \label{a5}
\\
&=\frac{\tilde{D}(x_0,r)}{rD(x_0,r)}-\frac{\tilde{H}(x_0,r)}{H(x_0,r)}+2 \frac{\hat{H}(x_0,r)H(x_0,r)-D(x_0,r)^2}{D(x_0,r)H(x_0,r)}.\nonumber
\end{align}
But from \eqref{a6} we have
\[
D(x_0,r)=\int_{\partial B(x_0,r)}\sigma(x) u(x)\partial _\nu u(x)dS(x).
\]
Then we find by applying Cauchy-Schwarz's inequality
\[
D(x_0,r)^2\le \left(\int_{\partial B(x_0,r)}\sigma(x) u^2(x)dS(x)\right)\left( \int_{\partial B(x_0,r)}\sigma(x) (\partial _\nu u)^2(x)dS(x) \right).
\]
That is
\begin{equation}\label{(1.1)}
D^2(x_0,r)\le H(x_0,r)\hat{H}(x_0,r).
\end{equation}
This and \eqref{a5} lead
\begin{equation}\label{a7}
\frac{\partial_rN(x_0,r)}{N(x_0,r)}\ge \frac{\tilde{D}(x_0,r)}{rD(x_0,r)}-\frac{\tilde{H}(x_0,r)}{H(x_0,r)}.
\end{equation}
On the other hand
\begin{equation}\label{a8}
\left|\tilde{H}(x_0,r)\right|\le \varkappa \|\nabla a \|_\infty H(x_0,r)\le  \varkappa^2 H(x_0,r),
\end{equation}
and similarly
\begin{equation}\label{a9}
\left|\tilde{D}(x_0,r)\right|\le  \varkappa^2 r D(x_0,r).
\end{equation}
In light of \eqref{a7}, \eqref{a8} and \eqref{a9}, we derive
\[
\frac{\partial_rN(x_0,r)}{N(x_0,r)}\ge -2\varkappa^2,
\]
that is to say
\[
\partial _r( e^{2\varkappa^2 r}N(x_0,r))\ge 0.
\]
Consequently
\[
N(x_0,r)\le e^{2\varkappa^2 (\delta -r)}N(x_0,\delta )\le e^{2\varkappa^2\delta}N(x_0,\delta),
\]
as expected.
\end{proof}

\subsection{Polynomial lower bound}

\begin{lemma}\label{lemma-a4}
There exist a universal constant $\varpi$ and two constants $c=c(n,\varkappa)>0$ and $0<\gamma =\gamma (n,\varkappa)<1$ so that if
\[
\mathcal{C}_0(h)=M\varpi\varkappa^4(1+\mathbf{d})\delta^{-1}e^{3\varkappa^2\delta+[2\ln(cM/\eta)/\gamma] e^{[6n|\ln \gamma|] \mathfrak{c}h }},\quad h>0,
\]
then
\bea
\|N(u)(x, \cdot)\|_{L^\infty (0,\delta)}\le \mathcal{C}_0(|x-x_0|/\delta ),
\eea
for any $u\in \mathscr{S}(\mathcal{O}, x_0, M,\eta ,\delta/3)$, where $\mathfrak{c}=\mathfrak{c}_{\mathcal{O}}$ is as in Lemma \ref{Glemma}.
\end{lemma}

\begin{proof}
Pick $x\in \mathcal O^{\delta}$. Then from Lemma \ref{lemmalb1}
\[
\|\nabla u\|_{L^2(B(x,\delta /3) )} \ge  e^{-[\ln(cM/\eta)/\gamma] e^{[6n|\ln \gamma|] \mathfrak{c}|x-x_0|/\delta }},
\]
for some constant $c=c(n,\varkappa)$ and $0<\gamma =\gamma (n,\varkappa))<1$.

On the other hand, we establish in a quite classical manner the following Caccioppoli's inequality
\[
\|\nabla u\|_{L^2(B(x,\delta /3) )}^2\le \frac{\varpi \varkappa^2 (1+\mathbf{d})}{\delta^2}\| u\|_{L^2(B(x,\delta ) )}^2,
\]
where $\varpi$ is a universal constant. Therefore
\begin{equation}\label{a12.0}
\| u\|_{L^2(B(x,\delta ) )}^2\ge \tilde{\mathcal{C}}_0(|x-x_0|/\delta),
\end{equation}
where
\begin{equation}\label{a12.1}
\tilde{\mathcal{C}}_0(h)=\frac{\delta^2}{\varpi \varkappa^2 (1+\mathbf{d})}e^{-[2\ln(cM/\eta)/\gamma] e^{[6n|\ln \gamma |\mathfrak{c}]h}},\quad h>0.
\end{equation}

Since $ K(u)(x,\delta )\ge \varkappa ^{-1}\| u\|_{L^2(B(x,\delta ) )}^2$, we find
\begin{equation}\label{a12}
K(u)(x,\delta )\ge   \frac{\delta^2}{\varpi \varkappa^3 (1+\mathbf{d})}e^{-[2\ln(cM/\eta)/\gamma] e^{[6n|\ln \gamma|] \mathfrak{c}|x-x_0|/\delta }}.
\end{equation}

In light of Lemma \ref{lemma-a3}, we derive from \eqref{a12}
\begin{equation}\label{a13}
H(u)(x,\delta )\ge \frac{\delta e^{-\varkappa ^2\delta}}{\varpi \varkappa^3 (1+\mathbf{d})}e^{-[2\ln(cM/\eta)/\gamma] e^{[6n|\ln \gamma|] \mathfrak{c}|x-x_0|/\delta }}.
\end{equation}

In light of Lemma \ref{lemma-bff}, we get
\[
N(x,r)\le \varkappa e^{2\varkappa^2\delta}\frac{ \| \nabla u\|_{L^2(\mathcal{O})}}{H(u)(x,\delta)},\quad 0<r < \delta,
\]

This inequality and \eqref{a13} give, where $c=c(n,\varkappa)$ is a constant, 
\[
N(x,r)\le M\varpi\varkappa^4(1+\mathbf{d})\delta^{-1}e^{3\varkappa^2\delta+[2\ln(cM/\eta)/\gamma] e^{[6n|\ln \gamma|] \mathfrak{c}|x-x_0|/\delta }}
,\quad 0<r < \delta,
\]
which is the expected inequality.
\end{proof}

\begin{proposition}\label{proposition-a2}
Let $\mathcal{C}_0$ be as in Lemma \ref{lemma-a4}, $\tilde{\mathcal{C}}_0$ as in \eqref{a12.1} and set
\begin{align}
&\mathcal{C}_1(h)=2\mathcal{C}_0(h)+n,\quad h>0,\label{a13.0}
\\
&\tilde{\mathcal{C}}_2(h)=\varkappa^{-2}e^{-\varkappa^2\delta}\tilde{\mathcal{C}}_0(h),\quad h>0.\label{a13.1}
\end{align} 
If $u\in \mathscr{S}(\mathcal{O}, x_0, M,\eta ,\delta/3)$ then
\[
\tilde{\mathcal{C}}_2(|x-x_0|/\delta)\left(\frac{r}{\delta}\right)^{\mathcal{C}_1(|x-x_0|/\delta)}\le \|u\|^2_{L^2(B(x, r))},\quad x\in \mathcal{O}^\delta ,\; 0<r <\delta .
\]
\end{proposition}

\begin{proof}
Observing that, where $H=H(u)$,
\[
\partial _r\left(\ln \frac{H(x,r)}{r^{n-1}}\right)=\frac{\partial _rH(x,r)}{H(x,r)}-\frac{n-1}{r},
\]
we get from Lemma \ref{lemma-a4}, \eqref{a1} and the fact that $|\tilde{H}(x,r)|\le \varkappa^2 H(x,r)$,
\[
\partial _r\left(\ln \frac{H(x,r)}{r^{n-1}}\right)\le   \varkappa^2+\frac{2N(x,r)}{r}\le   \varkappa^2+
\frac{2\mathcal{C}_0(|x-x_0|/\delta)}{r},\quad 0<r<\delta,
\]
Thus
\[
\int_{sr}^{s\delta} \partial _t\left(\ln \frac{H(x,t)}{t^{n-1}}\right)dt=\ln \frac{H(x,s\delta)r^{n-1}}{H(x,sr)\delta ^{n-1}} 
\le  \varkappa^2 (\delta -r)s+ 2\mathcal{C}_0(|x-x_0|/\delta)\ln \frac{\delta}{r},
\]
for $0<s<1$ and  $0<r<\delta$. Hence
\[
H(x,s\delta)\le  e^{\varkappa^2\delta}\left(\frac{\delta}{r}\right)^{\mathcal{C}_1(|x-x_0|/\delta)-1}H(x,sr),
\]
and then
\begin{align*}
\|u\|_{L^2(B(x,\delta))}^2&\le \varkappa \delta\int_0^1H(x,s\delta)ds
\\
&\le \varkappa \delta e^{\varkappa^2\delta}\left(\frac{\delta}{r}\right)^{\mathcal{C}_1(|x-x_0|/\delta)-1}\int_0^1H(x,rs )ds
\\
&\le \varkappa ^2e^{\varkappa^2\delta}\left(\frac{\delta}{r}\right)^{\mathcal{C}_1(|x-x_0|/\delta)} \|u\|_{L^2(B(x,r ))}^2.
\end{align*}
Combined with \eqref{a12.0} this estimate yields in a straightforward manner
\[
\varkappa^{-2}e^{-\varkappa^2\delta}\tilde{\mathcal{C}}_0(|x-x_0|/\delta)\left(\frac{r}{\delta}\right)^{\mathcal{C}_1(|x-x_0|/\delta)}\le \|u\|^2_{L^2(B(x, r))}.
\]
This is the expected inequality.
\end{proof}

For a bounded domain $D$, we denote the first non zero eigenvalue of the Laplace-Neumann operator on $D$ by $\mu _2(D)$. Since $\mu _2(B(x_0,r))=\mu_2(B(0,1))/r^2$, we get by applying Poincar\'e-Wirtinger's inequality
\begin{align}
\|w-\{w\}\|_{L^2(B(x,r))}^2 &\le \frac{1}{\mu _2(B(x,r))}\|\nabla w\|_{L^2(B(x,r))}^2 \label{a14}
\\
&\le \frac{r^2}{\mu _2(B(0,1))}\|\nabla w\|_{L^2(B(x,r))}^2,\nonumber
\end{align}
for any $w\in H^1(B(x,r))$, where $\{w\}=\frac{1}{|B(x,r)|}\int_{B(x,r)}w(x)dx$.
 
Noting that $\mathscr{S}(\mathcal{O}, x_0, M, \eta ,\delta/3)$ is invariant under the transformation $u\rightarrow u-\{u\}$, we can state the following consequence of Proposition \ref{proposition-a2}

\begin{corollary}\label{corollary-a1}
With the notations of Proposition \ref{proposition-a2}, if $u\in \mathscr{S}(\mathcal{O}, x_0, M,\eta ,\delta/3)$ then
\[
\mathcal{C}_2(|x-x_0|/\delta)\left(\frac{r}{\delta}\right)^{\mathcal{C}_1(|x-x_0|/\delta)}\le \|\nabla u\|^2_{L^2(B(x, r))},\quad x\in \mathcal{O}^\delta ,\; 0<r <\delta ,
\]
with 
\begin{equation}\label{a13.2}
\mathcal{C}_2(h)=\mu _2(B(0,1))\delta^{-2}\tilde{\mathcal{C}}_2(h),\quad h>0,
\end{equation}
with $\tilde{\mathcal{C}}_2$ as in Proposition \ref{proposition-a2}.
\end{corollary}

It is important to remark that the argument we used to obtain Corollary \ref{corollary-a1} from Proposition \ref{proposition-a2} is no longer valid if we substitute $L_\sigma$ by $L_\sigma$ plus a multiplication operator by a function $\sigma_0$.

The following consequence of the preceding corollary will be useful in the proof of Theorem \ref{theorem1}.

\begin{lemma}\label{lemma-a2}
Let $\omega \Subset \mathcal{O}$  and set $\delta=\mbox{dist}(\omega ,\partial \mathcal{O})$. Let $u\in \mathscr{S}(\mathcal{O}, x_0, M,\eta ,\delta/3)$ and $f\in C^{0,\alpha}(\overline{\mathcal{O}})$. Then we have
\begin{equation}\label{E2}
\|f\|_{L^\infty (\omega )}\le \hat{\mathcal{C}}_3 \|f\|_{C^{0,\alpha}(\overline{\mathcal{O}})}^{1-\hat{\mu}} \|f|\nabla u|^2\|_{L^1(\mathcal{O} )}^{\hat{\mu}},
\end{equation}
with 
\begin{align*}
&\hat{\mu} = \frac{\alpha}{\max_{x\in \overline{\mathcal{O}}}\mathcal{C}_1(|x-x_0|/\delta)+\alpha},
\\
&\hat{\mathcal{C}}_3=\max\left( 2\delta^{\alpha}(\max\left(1,(\hat{\mathcal{C}}_2\delta^\alpha)^{-1}\right), \max\left(1,M^2\right)(\hat{\mathcal{C}}_2\delta^\alpha)^{-1}\right),
\end{align*}
where $\hat{\mathcal{C}}_2=\max_{x\in \overline{\mathcal{O}}}\mathcal{C}_2(|x-x_0|/\delta)$ with $\mathcal{C}_2$ is as in Corollary \ref{corollary-a1}.
\end{lemma}

\begin{proof}
By homogeneity it is enough to consider those functions $f\in C^{0,\alpha}(\overline{\mathcal{O}})$ satisfying $\|f\|_{C^{0,\alpha}(\overline{\mathcal{O}})}=1$. Let $\mathcal{C}_1$ and $\mathcal{C}_2$ be respectively as in \eqref{a13.0} and \eqref{a13.2}. Let $u\in \mathscr{S}(\mathcal{O}, x_0, M,\eta ,\delta/3)$ and $f\in C^{0,\alpha}(\overline{\mathcal{O}})$ satisfying $\|f\|_{C^{0,\alpha}(\overline{\mathcal{O}})}=1$. Pick then ${x}\in \overline{\omega}$. From Corollary \ref{corollary-a1}, we have 
\begin{equation}\label{ii1}
\mathcal{C}_2(|x-x_0|/\delta )\left(\frac{r}{\delta}\right)^{\mathcal{C}_1(|x-x_0|/\delta )}\le \|\nabla u\|_{L^2(B(x, r ))}^2,\quad  0<r<\delta .
\end{equation}
On the other hand, it is straightforward to check that
\[
|f({x})|\le |f(y)|+r^\alpha ,\quad y\in B({x},r).
\]
Whence
\begin{align*}
|f(x)|\int_{B({x},r)}|\nabla u(y)|^2 dy\le \int_{B({x},r)}&|f(y)||\nabla u(y)|^2dy
\\
&+r^\alpha\int_{B({x},r)}|\nabla u(y)|^2dy.
\end{align*}
That is we have
\[
|f(x)| \|\nabla u\|^2_{L^2(B({x},r)}\le \|f|\nabla u|^2\|_{L^1(B({x},r ))}+r^\alpha \|\nabla u\|^2_{L^2(B({x},r))}.
\]
Since $u$ is non constant, by the unique continuation property, we have $\|\nabla u\|^2_{L^2(B({x},r))}\ne 0$, $0<r< \delta$. Therefore
\[
|f(x)|\le \frac{\|f|\nabla u|^2\|_{L^1(B({x},r ))}}{\|\nabla u\|^2_{L^2(B({x},r))}}+r^\alpha ,\quad 0<r< \delta .
\]
This and \eqref{ii1} entail
\[
|f(x)| \le \mathcal{C}_2(|x-x_0|/\delta )^{-1}\left(\frac{\delta}{r}\right)^{\mathcal{C}_1(|x-x_0|)}\|f|\nabla u|^2\|_{L^1(B(x,r ))}+r^\alpha ,\quad 0<r< \delta .
\]
Hence
\[
|f(x)| \le  \mathcal{C}_2(|x-x_0|/\delta )^{-1}\left(\frac{1}{s}\right)^{\mathcal{C}_1(|x-x_0|)}\|f|\nabla u|^2\|_{L^1(\mathcal{O})}+\delta ^\alpha s^\alpha ,\quad 0<s< 1 .
\]
In consequence
\[
\|f\|_{L^\infty(\omega)} \le \hat{\mathcal{C}}_2\left(\frac{1}{s}\right)^{\hat{\alpha}}\|f|\nabla u|^2\|_{L^1(\mathcal{O})}+\delta ^\alpha s^\alpha ,\quad 0<s< 1 ,
\]
where $\hat{\alpha}=\max_{x\in \overline{\mathcal{O}}}\mathcal{C}_1(|x-x_0|/\delta)$. The expected inequality follows by minimizing the right hand side of the last inequality, with respect to $s$.
\end{proof}

\section{Proof Theorem \ref{theorem1}}\label{section4}

Pick $(a,b), (\tilde{a},\tilde{b})\in \mathcal{D}(\lambda, \kappa)$ and let $u_j=G_{a,b}(\cdot ,\xi_j)$ and $\tilde{u}_j=G_{\tilde{a},\tilde{b}}(\cdot ,\xi_j)$, $j=1,2$. By simple computations we can check that  $w=u_2/u_1$ is the solution of the equation
\[
\mbox{div}(\sigma \nabla w)=0\quad \mbox{in}\; \mathbb{R}^n\setminus\{\xi_1,\xi_2\},
\]
with 
\[
\sigma =au_1^2=\frac{av_1^2}{b^2}.
\]
Similarly, $\tilde{w}=\tilde{u}_2/\tilde{u}_1$ is the solution of the equation
\[
\mbox{div}(\tilde{\sigma} \nabla \tilde{w})=0\quad \mbox{in}\; \mathbb{R}^n\setminus\{\xi_1,\xi_2\},
\]
with 
\[
\tilde{\sigma} =\tilde{a}\tilde{u}_1^2=\frac{\tilde{a}\tilde{v}_1^2}{\tilde{b}^2}.
\]
We know from Lemma \ref{lemmaq1} that there exist $x^\ast\in  B(\xi_2,|\xi_1-\xi_2|/2)\setminus\{\xi_2\}$, $\eta_0=\eta_0(n,\lambda,\kappa, |\xi_1-\xi_2|)>0$ and $\rho=\rho(n,\lambda,\kappa, |\xi_1-\xi_2|)>0$  so that $\overline{B}(x^\ast,\rho)\subset B(\xi_2,|\xi_1-\xi_2|/2)\setminus\{\xi_2\}$ and 
\begin{equation}\label{m1}
\eta_0\le \| \nabla w\|_{L^2(B(x^\ast,\rho))}.
\end{equation}
Fix then a bounded domain $\mathcal{Q}$ of $\mathbb{R}^n\setminus\{\xi_1,\xi_2\}$ is such a way that $\Omega \cup B(x^\ast,\rho)\Subset \mathcal{Q}$, and set
\[
\delta =\mbox{dist}(\Omega \cup B(x^\ast,\rho),\partial \mathcal{Q}).
\]
In the rest of this proof $\mathbf{d}=\mbox{diam}(\mathcal{Q})$. According to Corollary \ref{corollaryq2} 
\begin{equation}\label{m2}
\|\nabla w\|_{L^2(\mathcal{Q})}\le M=Ce^{c(\mathbf{d}+\varrho_+)}\left(1+ \max\left(\varrho_-^{-(2+\alpha)},1\right)\varrho_-^{-n+2}\right)^4,
\end{equation}
with $C=C(n,\lambda ,\kappa ,\alpha, \theta)$ and $c=c(n,\lambda,\kappa,\alpha,\theta)$, $\varrho_-=\min \left(\mathrm{dist}\left(\xi_1,\mathcal{Q}\right),\mbox{dist}\left(\xi_2,\mathcal{Q}\right)\right)$ and $\varrho_+=\max \left(\mathrm{dist}\left(\xi_1,\mathcal{Q}\right),\mathrm{dist}\left(\xi_2,\mathcal{Q}\right)\right)$.

Now, since 
\[
\|\sigma\|_{C^{0,1}(\overline{Q})}\le \|a\|_{C^{0,1}(\overline{Q})}\|u_1\|_{C^{0,1}(\overline{Q})}^2,
\]
we get, similarly to the end of the proof of Corollary \ref{corollaryq2}, from \cite[Lemma 6.35, page 135]{GT}
\[
\|\sigma\|_{C^{0,1}(\overline{Q})}\le C\|a\|_{C^{0,1}(\overline{Q})}\|u_1\|_{C^{2,\alpha}(\overline{Q})}^2, 
\]
where $C=C(n,\lambda ,\kappa, \mathbf{d} ,\xi_1,\xi_2)>0$ is a constant. This inequality together with Proposition \ref {propositionlr1} yield
\begin{equation}\label{m3.0}
\|\sigma\|_{C^{0,1}(\overline{Q})}\le C,
\end{equation}
for some constant $C=C(n,\lambda ,\kappa, \mathbf{d} ,\xi_1,\xi_2)>0$.

On the other hand, we have from \eqref{maineqq}
\begin{equation}\label{m3.1}
C^{-1}\min_{x\in \overline{\mathcal{Q}}}\frac{ e^{-2\sqrt{c \kappa}|x-\xi_1|}}{|x-\xi_1|^{n-2}}\le u_1,\quad \mbox{in}\; \overline{\mathcal{Q}},
\end{equation}
with constants $c=c(n,\lambda)>0$ and $C=C(n,\lambda,\kappa)>0$.

We get by combining \eqref{m3.0} and \eqref{m3.1} that there exists $\varkappa=\varkappa(n,\lambda ,\kappa ,\alpha,\Omega ,\xi_1,\xi_2)>1$ so that
\[
\varkappa ^{-1}\le \sigma \quad \mbox{and}\quad \|\sigma\|_{C^{0,1}(\overline{Q})}\le \varkappa .
\]
Next, if $\rho \le \delta /3$ then \eqref{m1} implies obviously 
\begin{equation}\label{m3}
\eta_0 \le \|\nabla w\|_{L^2(B(x_0,\delta/3))},
\end{equation}
with $\eta_0$ as in \eqref{m1}. When $\rho>\delta/3$ we can use the three-ball inequality in Theorem \ref{theoremlb1} in order to get 
\[
\tilde{C}\|\nabla w\|_{L^2(B(x^\ast,\rho))}\le \|\nabla w\|_{L^2(B(x_0,\delta/3))}^s\|\nabla w\|_{L^2(B(x^\ast,\rho+\delta/3))}^{1-s},
\]
where $\tilde{C}=\tilde{C}(n,\lambda ,\kappa ,\Omega ,\xi_1,\xi_2)$ and $0<s=s(n,\lambda ,\kappa ,\Omega ,\xi_1,\xi_2)<1$ are constants. Whence
\begin{equation}\label{m3.2}
(\tilde{C}\eta_0)^{1/s}M^{(s-1)/s}\le \|\nabla w\|_{L^2(B(x_0,\delta/3))}.
\end{equation}
In light of \eqref{m2}, \eqref{m3} and \eqref{m3.2}, we can infer that, for some constant $\eta =\eta (n,\lambda ,\kappa ,\Omega ,\xi_1,\xi_2)>0$,   $w\in \mathscr{S}(\mathcal{Q},x^\ast, M,\eta ,\delta /3)$, where $M$ is as in \eqref{m2} and $\mathscr{S}(\mathcal{Q},x^\ast, M,\eta ,\delta /3)$ is defined in  \eqref{S}.

\begin{lemma}\label{lemma6.1}
We have
\begin{equation}\label{6.2}
C\|(\sigma -\tilde{\sigma})|\nabla w |^2 \|_{L^1 (\Omega )}\le \|w -\tilde{w}\|_{L^2(\Omega )}^{\theta /(2+\theta)}
+\|\sigma -\tilde{\sigma} \|_{L^\infty (\Gamma )},
\end{equation}
where $C=C(n,\lambda ,\kappa ,\Omega ,\alpha,\theta,\xi_1,\xi_2)>0$ is a constant.
\end{lemma}

\begin{proof}
Clearly, if $\zeta =\sigma -\tilde{\sigma}$ and $u=w-\tilde{w}$, then
\[
\mathrm{div}(\tilde{\sigma} \nabla u )=\mathrm{div}(\zeta \nabla w).
\]
Recall that $\mathrm{sgn}_0$ is the sign function defined on $\mathbb{R}$ by: $\mathrm{sgn}_0(t)=-1$ if $t<1$, $\mathrm{sgn}_0(0)=0$ and $\mathrm{sgn}_0(t)=1$ if $t>0$. Since 
\begin{align*}
\mathrm{div}(|\zeta |\nabla w )&= \nabla |\zeta |\cdot \nabla w +|\zeta |\Delta w 
\\
&= \mathrm{sgn}_0(\zeta) \nabla \zeta \cdot \nabla w +\mathrm{sgn}_0 (\zeta)\zeta \Delta w
\\
&= \mathrm{sgn}_0(\zeta)\mathrm{div}(\zeta \nabla w )=\mathrm{sgn}_0(\zeta)\mathrm{div}(\tilde{\sigma} \nabla u ),
\end{align*}
we get by integrating by parts
\begin{align}
\int_\Omega |\zeta| |\nabla w |^2dx &=-\int_\Omega \mathrm{div}(|\zeta |\nabla w )w dx+ \int_\Gamma |\zeta |w \partial _\nu w dS(x)\label{eq5.1}
\\
&=-\int_\Omega  \mathrm{sgn}_0(\zeta)\mathrm{div}(\tilde{\sigma} \nabla u )w dx+\int_\Gamma |\zeta |w \partial _\nu w dS(x)  .\nonumber
\end{align}
Thus
\[
\int_\Omega |\zeta| |\nabla w |^2dx \le C\left( \|u\|_{H^2(\Omega )}+\|\zeta \|_{L^\infty (\Gamma )}\right).
\]
This, the following interpolation inequality 
\[
\|u\|_{H^2(\Omega )}\le c_\Omega \|u\|_{L^2(\Omega )}^{\theta /(2+\theta)}\| u\|_{H^{2+\theta}(\Omega )}^{2/(2+\theta)}
\]
and Corollary \ref{corollaryq2}  give \eqref{6.2}.
\end{proof}


We have from \eqref{E2} in Lemma \ref{lemma-a2}
\bea
\| \tilde \sigma-\sigma \|_{C(\overline{\Omega})} \le \hat{\mathcal{C}}_3
\|\tilde \sigma-\sigma\|_{ C^{0,\alpha}(\overline{\Omega})}^{1-\hat{\mu}}\|(\sigma -\tilde{\sigma})|\nabla w |^2\|_{L^1(\Omega )}^{\hat{\mu}},
\eea
from which we obtain
\bea
\| \tilde \sigma-\sigma \|_{C(\overline{\Omega})} \le \hat{\mathcal{C}}_3\max \left(1,
\|\tilde \sigma-\sigma\|_{C^{0,\alpha}(\overline{\Omega})}\right)\|(\sigma -\tilde{\sigma})|\nabla w |^2\|_{L^1(\Omega )}^{\hat{\mu}}.
\eea
Combined with Proposition \ref{propositionlr1},  this inequality gives
\bea
\| \tilde \sigma-\sigma \|_{C(\overline{\Omega})}\le C
\|(\sigma -\tilde{\sigma})|\nabla w |^2\|_{L^1(\Omega )}^{\hat{\mu}}.
\eea
Here and henceforward, $C=C(n,\lambda ,\kappa ,\Omega ,\alpha,\theta,\xi_1,\xi_2)>0$ is a generic constant.

Therefore, we obtain in light of Lemma \ref{lemma6.1}
\bea
 \| \tilde \sigma-\sigma \|_{C(\overline{\Omega})}\le C
\left(\|w -\tilde{w}\|_{L^2(\Omega )}^{\theta /(2+\theta)}+\|\sigma -\tilde{\sigma} \|_{C (\Gamma )}\right)
^{\hat{\mu}}.
\eea

Since $\tilde a= a$ and $\tilde b = b$ on $\Gamma$ and regarding the regularity of  $u_i$ and $\tilde u_i,\, i=1, 2$, we finally get 
\bean  \label{interm-inequality}
\| \tilde \sigma-\sigma \|_{C(\overline{\Omega})}\le C
\left(\|v_1 -\tilde{v}_1\|_{C(\overline{\Omega}) }+\|v_2 -\tilde{v}_2\|_{C(\overline{\Omega}) }\right)
^{\hat{\mu}_0},
\eean
with
\[
\hat{\mu}_0=\frac{\theta \hat{\mu}}{2+\theta}.
\]

The following lemma will be used in sequel.

\begin{lemma} \label{u1estim}
We have 
\begin{equation}\label{E0} 
\|u_1^{-1}-\tilde{u}_1^{-1}\|_{C^{2,\alpha}(\overline{\Omega})}  \le C
\left(\|v_1 -\tilde{v}_1\|_{C(\overline{\Omega}) }+\|v_2 -\tilde{v}_2\|_{C(\overline{\Omega}) }\right)
^{\hat{\mu}_1},
\end{equation}
where $0<\hat{\mu}_1=\hat{\mu}_1(n,\Omega ,\lambda,\kappa,\alpha,\theta,\xi_1,\xi_2)<1$ and $C=C(n,\Omega ,\lambda,\kappa,\alpha,\theta,\xi_1,\xi_2)>0$ are constants.
\end{lemma}

\begin{proof}
In this proof $C=C(n,\Omega ,\lambda,\kappa,\alpha,\theta,\xi_1,\xi_2)>0$  is a  generic constant. It is not hard to check that
\begin{align*}
&-\mathrm{div}(\sigma \nabla u_1^{-1})=v_1\quad \mathrm{in}\; \Omega ,
\\
&-\mathrm{div}(\tilde{\sigma} \nabla \tilde{u}_1^{-1})=\tilde{v}_1\quad \mathrm{in}\; \Omega .
\end{align*}
 Hence
\[
-\mathrm{div}(\sigma \nabla (u_1^{-1}-\tilde{u}_1^{-1}))=(v_1-\tilde{v}_1)+\mathrm{div}((\sigma -\tilde{\sigma})\nabla \widetilde{u}_1^{-1} )\quad \mathrm{in}\; \Omega.
\]
By the usual H\"older a priori estimate (see \cite[Theorem 6.6, page 98]{GT})
\begin{align*}
C\|u_1^{-1}-\tilde{u}_1^{-1}&\|_{C^{2,\alpha}(\overline{\Omega})}\le  \|v_1-\tilde{v}_1\|_{C^{0,\alpha}(\overline{\Omega})}
\\
&+\|\mathrm{div}((\sigma -\tilde{\sigma})\nabla \widetilde{u}_1^{-1} )\|_{C^{0,\alpha}(\overline{\Omega})}+ \|u_1^{-1}-\tilde{u}_1^{-1}\|_{C^{0,\alpha}(\Gamma)}.
\end{align*} 
Consequently
\begin{equation}\label{m4}
\|u_1^{-1}-\tilde{u}_1^{-1}\|_{C^{2,\alpha}(\overline{\Omega})}\le C\left( \|v_1-\tilde{v}_1\|_{C^{0,\alpha}(\overline{\Omega})}+\|\sigma -\tilde{\sigma}\|_{C^{1,\alpha}(\overline{\Omega})}\right),
\end{equation}
where we used 
\[
\|u_1^{-1}-\tilde{u}_1^{-1}\|_{C^{0,\alpha}(\Gamma)}=\|b(v_1^{-1}-\tilde{v}_1^{-1})\|_{C^{0,\alpha}(\Gamma)}.
\]
On the other hand, since 
\[
\|\sigma -\tilde{\sigma}\|_{C^{1,1}(\overline{\Omega})}\le C,\quad \|v_1-\tilde{v}_1\|_{C^{1,\alpha}(\overline{\Omega})}\le C
\]
and $\Omega$ is $C^{1,1}$,  we get again from the interpolation inequality in \cite[Lemma 6.35, page 135]{GT} 
\begin{equation}\label{m6.0}
\|\sigma -\tilde{\sigma}\|_{C^{1,\alpha}(\overline{\Omega})}\le C\|\sigma -\tilde{\sigma}\|_{C(\overline{\Omega})}^\tau ,\quad \|v_1-\tilde{v}_1\|_{C^{0,\alpha}(\overline{\Omega})}\le C\|v_1-\tilde{v}_1\|^\tau_{C(\overline{\Omega})},
\end{equation}
where $0<\tau=\tau(\Omega, \alpha) <1$ is a constant. Inequality \eqref{m6.0} in \eqref{m4} yields
\begin{equation}\label{m7}
\|u_1^{-1}-\tilde{u}_1^{-1}\|_{C^{2,\alpha}(\overline{\Omega})}\le C\left( \|v_1-\tilde{v}_1\|^\tau_{C(\overline{\Omega})}+\|\sigma -\tilde{\sigma}\|_{C(\overline{\Omega})}^\tau\right).
\end{equation}
On the other hand, we have from \eqref{interm-inequality} 
\bean  \label{E3}
 \| \tilde \sigma-\sigma \|_{C(\overline{\Omega})} \le C
\left(\|v_1 -\tilde{v}_1\|_{C(\overline{\Omega}) }+\|v_2 -\tilde{v}_2\|_{C(\overline{\Omega}) }\right)
^{\hat{\mu}_0}.
\eean
Whence, we get in light of inequalities \eqref{m7} and \eqref{E3}, where $\hat{\mu}_1=\tau \hat{\mu}_0$,
\[
\|u_1^{-1}-\tilde{u}_1^{-1}\|_{C^{2,\alpha}(\overline{\Omega})}  \le C
\left(\|v_1 -\tilde{v}_1\|_{C(\overline{\Omega}) }+\|v_2 -\tilde{v}_2\|_{C(\overline{\Omega})}\right)
^{\hat{\mu}_1}.
\]
This is the expected inequality.
\end{proof}

Also, since 
\[
\|\sigma -\tilde{\sigma}\|_{C^{1,1}(\overline{\Omega})}\le C,\quad \|v_1-\tilde{v}_1\|_{C^{2,\alpha}(\overline{\Omega})}\le C,
\]
we can proceed as in the preceding proof to get
\begin{equation}\label{m6}
\|\sigma -\tilde{\sigma}\|_{C^{1,\alpha}(\overline{\Omega})}\le C\|\sigma -\tilde{\sigma}\|_{C(\overline{\Omega})}^\tau ,\quad \|v_1-\tilde{v}_1\|_{C^{1,\alpha}(\overline{\Omega})}\le C\|v_1-\tilde{v}_1\|^\tau_{C(\overline{\Omega})},
\end{equation}
the constant $0<\tau=\tau (\Omega ,\alpha) <1$. But 
\begin{align*}
a-\tilde{a}=\sigma u_1^{-2} -\tilde{\sigma}\tilde{u}_1{^{-2}}&=(\sigma -\tilde{\sigma}) u_1^{-2}+ \tilde{\sigma}(u_1^{-2} -\tilde{u}_1^{-2})
\\
&=(\sigma -\tilde{\sigma}) u_1^{-2}+\tilde{\sigma}(u_1^{-1} +\tilde{u}_1^{-1})(u_1^{-1} -\tilde{u}_1^{-1}).
\end{align*} 
Hence
\begin{equation}\label{m5.1}
\|a -\tilde{a}\|_{C^{1,\alpha}(\overline{\Omega})}\le C\left( \|u_1^{-1} -\tilde{u}_1^{-1}\|_{C^{1,\alpha}(\overline{\Omega})}+\|\sigma -\tilde{\sigma}\|_{C^{1,\alpha}(\overline{\Omega})}\right).
\end{equation}
This inequality  together with \eqref{interm-inequality}, \eqref{E0} and \eqref{m6} imply
\begin{equation}\label{EQ1}
\|a -\tilde{a}\|_{C^{1,\alpha}(\overline{\Omega})}\le C\left( \|v_1-\tilde{v}_1\|_{C(\overline{\Omega})}+\|v_2-\tilde{v}_2\|_{C(\overline{\Omega})}\right)^{\hat{\mu}_1}.
\end{equation}
We proceed similarly for $b-\tilde{b}$. Since
\[
b-\tilde{b}=v_1u_1^{-1}-\tilde{v}_1\tilde{u}_1^{-1}=(v_1-\tilde{v}_1)u_1^{-1}+\tilde{v}_1(u_1^{-1}-\tilde{u}_1^{-1}),
\]
we have
\begin{equation}\label{EQ2}
\|b -\tilde{b}\|_{C^{0,\alpha}(\overline{\Omega})}\le C\left( \|v_1-\tilde{v}_1\|_{C(\overline{\Omega})}+\|v_2-\tilde{v}_2\|_{C(\overline{\Omega})}\right)^{\hat{\mu}_1}.
\end{equation}
The expected inequality follows by putting together \eqref{EQ1} and \eqref{EQ2}.

\appendix
\section{Proof of technical lemmas}\label{appendixA}

\begin{proof}[Proof of Lemma \ref{lemma-tse}]
In this proof $C=C(n,\mu,\nu)>1$ is a generic constant.

It is well known that $G_{1,\nu}$, $\nu >0$, the fundamental solution of the operator $-\Delta +\nu$, 
 is given by $G_{1,\nu}(x,\xi )=\mathcal{G}_{1,\nu}(x-\xi)$, $x,\xi \in \mathbb{R}^n$, with
\[
\mathcal{G}_{1,\nu}(x)= (2\pi)^{-n/2}(\sqrt{\nu}/|x|)^{n/2-1}K_{n/2-1}(\sqrt{\nu}|x|).
\]
In the particular case $n=3$, we have $K_{1/2}(z)=\sqrt{\pi /(2z)}e^{-z}$ and therefore
\[
\mathcal{G}_{1,\nu}(x)=\frac{e^{-\sqrt{\nu}|x|}}{4\pi |x|}.
\]

Let $f\in C_0^\infty (\mathbb{R}^n)$,    $\mu>0$ and $\nu >0$ be two constants, and denote by $u$ the solution of the equation
\[
(-\mu\Delta +\nu)u=f\quad \mbox{in}\; \mathbb{R}^n.
\]
Then 
\begin{equation}\label{app1}
u(x)=\int_{\mathbb{R}^n}G_{\mu,\nu}(x,\xi) f(\xi)d\xi ,\quad x\in \mathbb{R}^n.
\end{equation}
We remark that $v(x)=u(\sqrt{\mu}x)$, $x\in \mathbb{R}^n$ satisfies $(-\Delta +\nu )v=f(\sqrt{\mu}\; \cdot )$. Whence
\begin{align*}
u(\sqrt{\mu}x)=v(x)&= \int_{\mathbb{R}^n}\mathcal{G}_{1,\kappa}(x-\xi) f(\sqrt{\mu}\xi)d\xi 
\\
&= \mu^{-n/2}\int_{\mathbb{R}^n}\mathcal{G}_{1,\nu}(x-\xi/\sqrt{\mu}) f(\xi)d\xi ,\quad x\in \mathbb{R}^n.
\end{align*}
Hence
\begin{equation}\label{app2}
u(x)=\mu^{-n/2}\int_{\mathbb{R}^n}\mathcal{G}_{1,\nu}((x-\xi)/\sqrt{\mu}) f(\xi)d\xi ,\quad x\in \mathbb{R}^n.
\end{equation}
Comparing \eqref{app1} and \eqref{app2} we find
\[
G_{\mu ,\nu}(x,\xi)=\mu^{-n/2}\mathcal{G}_{1,\nu}((x-\xi)/\sqrt{\mu}),\quad x,\xi\in \mathbb{R}^n.
\]
Consequently $G_{\mu ,\nu}(x,\xi)=\mathcal{G}_{\mu ,\nu}(x-\xi)$ with
\begin{equation}\label{app3}
\mathcal{G}_{\mu,\nu}(x)= (2\pi \mu)^{-n/2}(\sqrt{\nu \mu}/|x|)^{n/2-1}K_{n/2-1}(\sqrt{\nu}|x|/\sqrt{\mu}),\quad x\in \mathbb{R}^n.
\end{equation}
By the usual asymptotic formula for modified Bessel functions of the second kind (see for instance \cite[9.7.2, page 378]{AS}) we have, when $|x|\rightarrow \infty$,
\[
K_{n/2-1}(\sqrt{\nu}|x|/\sqrt{\mu})=\left(\frac{\pi \sqrt{\mu}}{2\sqrt{\nu}|x|}\right)^{1/2}e^{-\sqrt{\nu}|x|/\sqrt{\mu}}\left( 1+O(1/|x|)\right),
\]
where $O(1/|x|)$ only depends on $n$, $\mu$ and $\nu$.

Consequently, there exits $R=R(n,\mu,\nu)>0$  so that
\begin{equation}\label{app4}
C^{-1}\frac{e^{-\sqrt{\nu}|x|/\sqrt{\mu}}}{|x|^{1/2}}\le K_{n/2-1}(\sqrt{\nu}|x|/\sqrt{\mu})\le C\frac{e^{-\sqrt{\nu}|x|/\sqrt{\mu}}}{|x|^{1/2}},\quad |x|\ge R.
\end{equation}
Substituting if necessary  $R$ by $\max (R,1)$, we have
\begin{equation}\label{app5}
\frac{1}{|x|^{n/2-1}}\le \frac{1}{|x|^{1/2}},\quad |x|\ge R.
\end{equation}
Moreover, we have 
\[
\frac{e^{-\sqrt{\nu}|x|/\sqrt{\mu}}}{|x|^{1/2}}=\left[|x|^{(n-3)/2}e^{-\sqrt{\nu}|x|/(2\sqrt{\mu})}\right]\frac{e^{-\sqrt{\nu}|x|/(2\sqrt{\mu})}}{|x|^{n/2-1}},\quad |x|\ge R.
\]
Since the function $x\rightarrow |x|^{(n-3)/2}e^{-\sqrt{\nu}|x|/(2\sqrt{\mu})}$ is bounded in $\mathbb{R}^n$, we deduce
\begin{equation}\label{app6}
\frac{e^{-\sqrt{\nu}|x|/\sqrt{\mu}}}{|x|^{1/2}}\le C\frac{e^{-\sqrt{\nu}|x|/(2\sqrt{\mu})}}{|x|^{n/2-1}},\quad |x|\ge R.
\end{equation}
Using \eqref{app5} and \eqref{app6} in \eqref{app4} in order to obtain
\begin{equation}\label{app7}
C^{-1}\frac{e^{-\sqrt{\nu}|x|/\sqrt{\mu}}}{|x|^{n/2-1}}\le K_{n/2-1}(\sqrt{\nu}|x|/\sqrt{\mu})\le C\frac{e^{-\sqrt{\nu}|x|/(2\sqrt{\mu})}}{|x|^{n/2-1}},\quad |x|\ge R.
\end{equation}

We now establish a similar estimate when $|x|\rightarrow 0$. To this end, we recall that according to formula \cite[9.6.9, page 375]{AS} we have
\[
K_{n/2-1}(\rho)\sim \frac{1}{2}\Gamma (n/2-1)\left(\frac{2}{\rho}\right)^{n/2-1}\quad \mbox{as}\; \rho\rightarrow 0,
\]
from which we deduce in a straightforward manner that there exists $0<r\le R$ so that
\begin{equation}\label{app8}
C^{-1}\frac{e^{-\sqrt{\nu}|x|/\sqrt{\mu}}}{|x|^{n/2-1}}\le K_{n/2-1}(\sqrt{\nu}|x|/\sqrt{\mu})\le C\frac{e^{-\sqrt{\nu}|x|/(2\sqrt{\nu})}}{|x|^{n/2-1}},\quad |x|\le r.
\end{equation}
The expected two sided inequality \eqref{tse} follows by combining \eqref{app4}, \eqref{app7} and \eqref{app8}.
\end{proof}

\begin{proof}[Proof of Lemma \ref{lemmaA1}]
Let $\mathcal{Q}$ be an open subset of $\mathbb{R}^n$, set $d=\mbox{diam}(\mathcal{Q})$, $d_x=\mbox{dist}(x,\partial \mathcal{Q})$ and $d_{x,y}=\min (d_x,d_y)$.

We introduce the following weighted H\"older semi-norms and  H\"older norms, where $\sigma \in \mathbb{R}$, $0<\gamma \le 1$, and $k$ is non-negative integer,
\begin{align*}
&[w]_{k,0;\mathcal{Q}}^{(\sigma)}=[w]_{k,\mathcal{Q}}^{(\sigma)}=\sup_{x\in \mathcal{Q},\; |\beta |=k}d_x^{k+\sigma}|\partial ^\beta w(x)|,
\\
&[w]_{k,\gamma;\mathcal{Q}}^{(\sigma)}=\sup_{x,y\in \mathcal{Q},\; |\beta |=k}d_{x,y}^{k+\gamma+\sigma}\frac{|\partial ^\beta w(y)-\partial ^\beta w(x)|}{|y-x|^\gamma},
\\
&|w|_{k;\mathcal{Q}}^{(\sigma)}=\sum_{j=0}^k[w]_{j;\mathcal{Q}}^{(\sigma)},
\\
&|w|_{k,\gamma;\mathcal{Q}}^{(\sigma)}=|w|_{k;\mathcal{Q}}^{(\sigma)}+[w]_{k,\gamma;\mathcal{Q}}^{(\sigma)}.
\end{align*}
In term of these notations, we have
\begin{align*}
&|a|_{0,\alpha ;\mathcal{Q}}^{(0)}=\sup_{x\in \mathcal{Q}}|a(x)|+\sup_{x,y\in \mathcal{Q}}d_{x,y}^{\alpha}\frac{|a(y)- a(x)|}{|y-x|^\alpha}\le (1+\mathbf{d})\lambda ,
\\
&|\partial_ja|_{0,\alpha ;\mathcal{Q}}^{(1)}=\sup_{x\in \mathcal{Q}}d_x|\partial_j a(x)|+\sup_{x,y\in \mathcal{O}}d_{x,y}^{1+\alpha}\frac{|\partial_ja(y)- \partial_ja(x)|}{|y-x|^\alpha}\le (\mathbf{d}+\mathbf{d}^2)\lambda,
\\
&|b|_{0,\alpha ;\mathcal{Q}}^{(2)}=\sup_{x\in \mathcal{O}}d_x^2|b(x)|+\sup_{x,y\in \mathcal{Q}}d_{x,y}^{2+\alpha}\frac{|b(y)- b(x)|}{|y-x|^\alpha}\le (\mathbf{d}^2+\mathbf{d}^3)\lambda .
\end{align*}
In consequence
\begin{equation}\label{A1}
|a|_{0,\alpha ;\mathcal{Q}}^{(0)}+|\partial_ja|_{0,\alpha ;\mathcal{Q}}^{(1)}+|b|_{0,\alpha ;\mathcal{Q}}^{(2)}\le \Lambda (\mathbf{d})=[1+2\mathbf{d}+2\mathbf{d}^2+\mathbf{d}^3]\lambda .
\end{equation}
Following \cite{GT} we define also
\begin{align*}
&[w]_{k,0;\mathcal{Q}}^\ast=[w]_{k,\mathcal{O}}^\ast=\sup_{x\in \mathcal{Q},\; |\beta |=k}d_x^{k}|\partial ^\beta w(x)|,
\\
&[w]_{k,\gamma;\mathcal{Q}}^\ast=\sup_{x,y\in \mathcal{Q},\; |\beta |=k}d_{x,y}^{k+\alpha}\frac{|\partial ^\beta w(y)-\partial ^\beta w(x)|}{|y-x|^\gamma},
\\
&|w|_{k;\mathcal{Q}}^\ast=\sum_{j=0}^k[w]_{j;\mathcal{Q}}^\ast,
\\
&|w|_{k,\gamma;\mathcal{Q}}^\ast=|w|_{k;\mathcal{Q}}^\ast+[w]_{k,\gamma;\mathcal{O}}^\ast.
\end{align*}

From \cite[Lemma 6.32, page 130]{GT} and its proof we have the following interpolation inequalities: suppose that $j$ and $k$, non negative integers, and $0\le \beta,\gamma \le 1$ are so that $j+\beta <k+\gamma$. Then there exist $C=C(n,\alpha,\beta )>0$ and $\vartheta =\vartheta (\alpha ,\beta )$ so that, for any $w\in C^{k,\alpha}(\mathcal{Q})$ and $\epsilon >0$, we have
\begin{align}
&[w]_{j,\beta;\mathcal{Q}}^\ast\le C\epsilon^{-\vartheta}|w|_{0;\mathcal{Q}}+\epsilon [w]_{k,\gamma;\mathcal{Q}}^\ast \label{A2},
\\
&|w|_{j,\beta;\mathcal{Q}}^\ast\le C\epsilon^{-\vartheta}|w|_{0;\mathcal{Q}}+\epsilon [w]_{k,\gamma;\mathcal{Q}}^\ast \label{A3}.
\end{align}
Here $|w|_{0;\mathcal{Q}}=\sup_{x\in \mathcal{Q}}|w(x)|$.

Checking carefully the proof of interior Schauder estimates in \cite[Theorem 6.2, page 90]{GT}, we get, taking into account inequalities \eqref{A1}-\eqref{A3}, the following result: there exist a constant $C=C(n)>0$ and $\tau =\tau (\alpha)$ so that, for any $0<\mu \le 1/2$ and $w\in C^{k,\alpha}(\mathcal{Q})$ satisfying $L_{a,b}w=0$ in $\mathcal{Q}$, we have
\begin{equation}\label{A4}
[w]_{2,\alpha,\mathcal{Q}}^\ast \le C\Lambda (\mathbf{d})\left(\mu^{-\tau}|w|_{0;\mathcal{Q}}+\mu^\alpha [w]_{2,\alpha,\mathcal{Q}}^\ast \right).
\end{equation}
Substituting in \eqref{A4} $C$ by $\max(C,2^{\alpha -1})$,  we may assume  in \eqref{A4} that $C=C(n,\alpha)\ge 2^{\alpha -1}$. Bearing in mind that $\Lambda (\mathbf{d})>1$, we can take in \eqref{A4}, $\mu =(2C\Lambda(\mathbf{d}))^{-1/\alpha} $. We find
\begin{equation}\label{A5}
[w]_{2,\alpha,\mathcal{Q}}^\ast \le C\Lambda (\mathbf{d})^\varkappa |w|_{0;\mathcal{Q}},
\end{equation}
for some constants $C=C(n,\alpha)>0$ and $\varkappa =\varkappa (\alpha )>1$.

Using again interpolation inequalities \eqref{A2} and \eqref{A3}, we deduce that 
\begin{equation}\label{A5}
|w|_{2,\alpha,\mathcal{Q}}^\ast \le C\Lambda (\mathbf{d})^\varkappa |w|_{0;\mathcal{Q}}.
\end{equation}

Let $\delta >0$ be so that $\mathcal{Q}_\delta =\{x\in \mathcal{Q};\; \mbox{dist}(x,\partial \mathcal{Q})>\delta\}$ is nonempty. If $\mathcal{Q}'$ is an open subset of $\mathcal{Q}_\delta$ then \eqref{A5} yields in a straightforward manner
\[
\|w\|_{C^{2,\alpha}\left(\overline{\mathcal{Q}'}\right)} \le C\max\left(\delta^{-(2+\alpha)},1\right)\Lambda (\mathbf{d})^\varkappa |w|_{0;\mathcal{Q}}.
\]
This is the expected inequality.
\end{proof}

\begin{lemma}\label{lemmaB1}
Let $K$ be a compact subset of $\mathbb{R}^n$ and $f\in C^{2,\alpha}(K)$ satisfying $\min_K|f|\ge c_->0$. Then
\begin{equation}\label{aB0}
\|1/f\|_{C^{2,\alpha}(K)}\le  C c_+^4\left(1+\|f\|_{C^{2,\alpha}(K)}\right)^3,
\end{equation}
where $c_+=\max (1,c_-^{-1})$ and $C=C(\mathrm{diam}(K))$ is a constant.
\end{lemma}

\begin{proof}
Let $x,y\in K$. Using  $|1/f|_{0;K}\le c_+$ and the following identities
\begin{align*}
&\frac{1}{f^2(y)}-\frac{1}{f^2(x)}=\left(\frac{1}{f(x)f^2(y)}+\frac{1}{f(x)^2f(y)}\right) (f(x)-f(y)),
\\
&\frac{1}{f^3(y)}-\frac{1}{f^3(x)}=\left(\frac{1}{f(x)f^3(y)}+\frac{1}{f^2(x)f^2(y)}+\frac{1}{f(x)^3f(y)}\right) (f(x)-f(y)),
\end{align*}
we easily get
\begin{equation}\label{aB1}
[1/f^j]_{\alpha ;K}\le 3c_+^4[f]_{\alpha ;K},\quad j=2,3.
\end{equation}
Also, we have
\begin{align*}
&\frac{\partial_if(y)\partial_j f(x)}{f^3(y)}-\frac{\partial_if(y)\partial_j f(x)}{f^3(x)}=\frac{\partial_if(y)}{f^3(y)}(\partial_j f(y)-\partial_j f(x))
\\
&\hskip 2cm + \frac{\partial_jf(x)}{f^3(y)}(\partial_i f(y)-\partial_i f(x))+\left(\frac{1}{f^3(y)}-\frac{1}{f^3(x)}\right)(\partial_if(y)\partial_j f(x)).
\end{align*}
In light of \eqref{aB1}, this identity yields
\begin{align}
&[\partial_if\partial_jf/f^3]_{\alpha ;K}\le c_+^4\left([\partial_if]_{\alpha ;K}|\partial_jf|_{0;K}\right.\label{aB2}
\\
&\hskip 4cm \left.+[\partial_jf]_{\alpha ;K}|\partial_if|_{0;K}+[f]_{\alpha ;K}|\partial_if|_{0;K}|\partial_jf|_{0;K}\right).\nonumber
\end{align}
On the other hand, since
\[
\frac{\partial_{ij}^2f(y)}{f^2(y)}-\frac{\partial_{ij}^2f(x)}{f^2(x)}=\frac{1}{f^2(y)}(\partial_{ij}^2f(y)-\partial_{ij}^2f(x))+\left(\frac{1}{f^2(y)}-\frac{1}{f^2(y)}\right)\partial_{ij}^2f(x),
\]
we find, by using again \eqref{aB1},
\begin{equation}\label{aB3}
[\partial_{ij}^2f/f^2]_{\alpha;K}\le 3c_+^4\left([\partial_{ij}^2f]_{\alpha;K}+[f]_{\alpha ;K}|\partial_{ij}^2f|_{0,K}\right).
\end{equation}
Inequalities \eqref{aB2}, \eqref{aB3}, the identity $\partial_{ij}^2(1/f)=2\partial_if\partial_jf/f^3-\partial_{ij}^2f/f^2$ and the interpolation inequality \cite[Lemma 6.35, page 135]{GT} (by proceeding as in Corollary \ref{corollaryq1}) imply 

\begin{equation}\label{aB4}
[\partial_{ij}^2(1/f)]_{\alpha ,K}\le  Cc_+^4\left(1+\|f\|_{C^{2,\alpha}(K)}\right)^3,
\end{equation}
with $C=C(\mbox{diam}(K))$ is a constant.

The other terms for $1/f$ appearing in the norms $\|\cdot \|_{C^{2,\alpha}(K)}$ can be estimated similarly to the semi-norm in \eqref{aB4}. Inequality \eqref{aB0} then follows.
\end{proof}

Recall that $0<\theta <\alpha <1$.

\begin{lemma}\label{lemmaB2}
$C^{2,\alpha}(\overline{\mathcal{O}})$ is continuously embedded in $H^{2+\theta}(\mathcal{O})$. Furthermore, there exists $C=C(n,\alpha -\theta)$ so that, for any $w\in C^{2,\alpha}(\overline{\mathcal{O}})$, we have
\begin{equation}\label{aB5}
\|w\|_{H^{2+\theta}(\mathcal {O})} \le C\max\left(\mathbf{d}^{n/2},\mathbf{d}^{n/2+\alpha -\theta}\right)\|w\|_{C^{2,\alpha}(\overline{\mathcal{O}})},
\end{equation}
where $\mathbf{d}=\mathrm{diam}(\mathcal{O})$.
\end{lemma}
\begin{proof}
Let $w\in C^{2,\alpha}(\overline{\mathcal{O}})$ and, for fixed $1\le i,j\le n$, set $g=\partial_{ij}^2w$. Then
\[
\int_{\mathcal{O}} \int_{\mathcal{O}} \frac{|g(x)-g(y)|^2}{|x-y|^{n+2\theta}}dxdy\le [g]_{\alpha ;\mathcal{O}}^2 \int_{\mathcal{O}} \int_{\mathcal{O}} \frac{1}{|x-y|^{n-2(\alpha -\theta)}}dxdy.
\]
In light of \cite[Lemma A3, page 246]{Ch}, this inequality yields
\[
\int_{\mathcal{O}} \int_{\mathcal{O}} \frac{|g(x)-g(y)|^2}{|x-y|^{n+2\theta}}dxdy\le \frac{|\mathbb{S}^{n-1}||\mathcal{O}|\mathbf{d}^{2(\alpha -\theta)}}{2(\alpha -\theta)}[g]_{\alpha ;\mathcal{O}}^2,
\]
But $|\mathcal{O}|\le |B(0,\mathbf{d})|$. Hence
\begin{equation}\label{aB6}
\int_{\mathcal{O}} \int_{\mathcal{O}} \frac{|g(x)-g(y)|^2}{|x-y|^{n+2\theta}}dxdy\le \frac{|\mathbb{S}^{n-1}|^2\mathbf{d}^{n+2(\alpha -\theta)}}{2(\alpha -\theta)}[g]_{\alpha ;\mathcal{O}}^2.
\end{equation}
Using \eqref{aB6} and the inequality
\[
\|h\|_{L^2(\mathcal{O})}^2\le |\mathbb{S}^{n-1}|\mathbf{d}^n|h|_{0,\mathcal{O}},\quad h\in C(\overline{\mathcal{O}}),
\]
we get from the definition of the norm of $H^s$-spaces in \cite[formula (1.3.2.2), page 17]{Gr}
\[
\|w\|_{H^{2+\theta}(\mathcal {O})} \le C\max\left(\mathbf{d}^{n/2},\mathbf{d}^{n/2+\alpha -\theta}\right)\|w\|_{C^{2,\alpha}(\overline{\mathcal{O}})},
\]
for some constant $C=C(n,\alpha -\theta)>0$. This is the expected inequality
\end{proof}

\end{document}